\documentclass[11pt, reqno]{amsart}
\usepackage{amsthm,amssymb}
\usepackage[foot]{amsaddr}

\newtheorem{lemma}{Lemma}[section]

\newtheorem{theorem}{Theorem}
\newtheorem{corollary}{Corollary}
\newtheorem{proposition}{Proposition}

\theoremstyle{definition}

\newtheorem*{definition*}{Definition}

\newtheorem*{examples}{Examples}
\newtheorem{remark}{Remark}
\newtheorem*{remark*}{Remark}
\newtheorem*{remarks}{Remarks}

\newcommand{\loc}{{\rm loc}}

\newcommand{\Real}{{\rm Re\,}}

\expandafter\def\expandafter\normalsize\expandafter{%
    \normalsize
    \setlength\abovedisplayshortskip{8pt}
    \setlength\belowdisplayshortskip{8pt}
}

\begin{document}

\title{On admissible singular drifts of symmetric $\alpha$-stable process}

\author{D.\,Kinzebulatov and K.\,R.\,Madou}

\begin{abstract}We consider the problem of existence of a (unique) weak solution to the SDE describing symmetric $\alpha$-stable process with a locally unbounded drift $b:\mathbb R^d \rightarrow \mathbb R^d$, $d \geq 3$, $1<\alpha<2$.
In this paper, $b$ belongs to the class of weakly form-bounded vector fields. The latter arises as the class providing the $L^2$ theory of the non-local operator behind the SDE, i.e.\,$(-\Delta)^{\frac{\alpha}{2}} + b \cdot \nabla$, and contains as proper sub-classes the other classes of singular vector fields studied in the literature in connection with this operator, such as the Kato class, weak $L^{\frac{d}{\alpha-1}}$ class and the Campanato-Morrey class (thus, $b$ can be so singular that it destroys the standard heat kernel estimates in terms of the heat kernel of the fractional Laplacian). We show that for such $b$ the operator $-(-\Delta)^{\frac{\alpha}{2}} - b \cdot \nabla$ admits a realization as a Feller generator, and that the probability measures determined by the Feller semigroup (uniquely in appropriate sense) admit description as weak solutions to the corresponding SDE. The proof is based on detailed regularity theory of $(-\Delta)^{\frac{\alpha}{2}} + b \cdot \nabla$  in $L^p$, $p>d-\alpha+1$.
\end{abstract}

\address{Universit\'{e} Laval, D\'{e}partement de math\'{e}matiques et de statistique, 1045 av.\,de la M\'{e}decine, Qu\'{e}bec, QC, G1V 0A6, Canada}

\email{damir.kinzebulatov@mat.ulaval.ca}

\address{Universit\'{e} Laval, D\'{e}partement de math\'{e}matiques et de statistique, 1045 av.\,de la M\'{e}decine, Qu\'{e}bec, QC, G1V 0A6, Canada}

\email{kodjo-raphael.madou.1@ulaval.ca}



\maketitle

\section{Introduction and main results}

\label{intro_sect}

Let $Z_t$ be a rotationally symmetric $\alpha$-stable process in $\mathbb{R}^{d}$, $d \geq 3$, $1<\alpha<2$, i.e.\,a L\'{e}vy process with characteristic function 
\begin{displaymath}
\mathbb{E}[\exp (i\varkappa \cdot (Z_{t}-Z_0)]=\exp  (-t|\varkappa |^{\alpha}) \quad \text{ for every } \varkappa \in \mathbb{R}^{d}.
\end{displaymath}
The (minus) generator of $Z_t$ is the fractional Laplace operator $(-\Delta)^{\frac{\alpha}{2}}$ given on $C_c^\infty$ by the formula
\begin{equation*}
(-\Delta)^{\frac{\alpha}{2}}f(x)=\lim_{\epsilon \downarrow 0} c\int_{|y|>\varepsilon}\frac{f(x+y)-f(x)}{|y|^{d+\alpha}}dy,\quad \text{ where } c:=\frac{\alpha 2^{\alpha-1}\Gamma(\frac{d+\alpha}{2})}{\pi^{\frac{d}{2}}\Gamma(\frac{2-\alpha}{2})}.
\end{equation*}

Let $b:\mathbb R^d \rightarrow \mathbb R^d$ be a measurable vector field with entries in $L^1_{\loc}\equiv L^1_{\loc}(\mathbb R^d)$.
The subject of this paper is the stochastic differential equation
\begin{equation}
\label{eq0}
X_t=x-\int_0^t b(X_s)ds + Z_t-Z_0, \quad t \geq 0, \quad x \in \mathbb R^d,
\end{equation} 
Recall that a weak solution to \eqref{eq0} is a process $X_t$ defined on some probability space having a.s.\,right continuous trajectories with left limits, such that $\int_0^t |b(X_s)|ds<\infty$ a.s.\,for every $t>0$, and such that $X_t$ satisfies \eqref{eq0} a.s.\,for a symmetric $\alpha$-stable process $Z_t$. 
A weak solution to \eqref{eq0}, when it exists (e.g.\,if $|b| \in L^\infty$, see \cite{Ko}), is called $\alpha$-stable process with drift $b$. It plays central role in the study of jump processes which, in contrast to diffusion processes, can have long range interactions.
The operator behind SDE \eqref{eq0} is the non-local operator
$(-\Delta)^{\frac{\alpha}{2}} + b \cdot \nabla$, i.e.\,one expects that the transition density of $X_t$ solves the corresponding parabolic equation for $(-\Delta)^{\frac{\alpha}{2}} + b \cdot \nabla$.

The following question is important: what are the minimal assumptions on the local singularities of the vector field $b$, not assuming additional structure such as symmetry or existence of the divergence, such that, for an arbitrary starting point, there exists a unique (in appropriate sense) weak solution to \eqref{eq0}? This question has been extensively studied in the literature. By the result in \cite{PP,P}, if 
\begin{equation}
\label{b_p}
|b| \in L^p+L^\infty, \quad \text{ for some } p>\frac{d}{\alpha-1}, 
\end{equation}
then there exists a unique in law weak solution to \eqref{eq0}. Although the exponent $\frac{d}{\alpha-1}$ is the best possible, the class \eqref{b_p} is far from being the maximal admissible: this result has been strengthened in \cite{CKS, CW, KS} where the authors consider $b$ in the standard Kato class $\mathbf{K}^{d,\alpha-1}_0$ containing, for a given $\varepsilon>0$, vector fields $b$ with $|b| \not \in L^{1+\varepsilon}_{\loc}$ (see more detailed discussion of the existing results below).
Similarly to these works, in this paper we search for the integral characteristics of $b$ that determines whether a unique weak solution to \eqref{eq0} exists. We consider the following larger class of vector fields:

\begin{definition*}
A vector field $b:\mathbb R^d \rightarrow \mathbb R^d$ with entries in $L^1_{\loc}\equiv L^1_{\loc}(\mathbb R^d)$ (we write $b \in L^1_{\loc}(\mathbb R^d,\mathbb R^d)$) is said to belong to the class of weakly form-bounded vector fields $\mathbf{F}_{\delta}^{\scriptscriptstyle \frac{\alpha-1}{2}}$, $\delta>0$ if  there exists $\lambda = \lambda_{\delta} > 0$ such that
$$
\big\| |b|^{\frac{1}{2}}\big(\lambda+(-\Delta)^{\frac{\alpha}{2}}\big)^{-\frac{\alpha-1}{2\alpha}}\big\|_{2 \rightarrow 2} \leq \sqrt{\delta}
$$
Here and below, $\|\cdot\|_{p \rightarrow q}$ denotes the $\|\cdot\|_{L^p \rightarrow L^q}$ operator norm.
\end{definition*}

Equivalently, 
$$
|b|  \leq \delta \big(\lambda+(-\Delta)^{\frac{\alpha}{2}}\big)^{\frac{\alpha-1}{\alpha}} \quad \text{ in the sense of quadratic forms}.
$$

Clearly, the sum of two weakly form-bounded vector fields is also weakly form-bounded (with different $\delta$). The constant $\delta$ is called the weak form-bound of $b$. It measures the size of critical singularities of the drift $b$: below we show that there is a quantitative dependence between the value of $\delta$ and the regularity properties of solutions to the corresponding elliptic and parabolic equations.

Our assumptions concerning $\delta$ will involve only strict inequalities, so using the Spectral Theorem we can re-state our hypothesis on $b$, i.e.\,$b \in \mathbf{F}_{\delta}^{\scriptscriptstyle \frac{\alpha-1}{2}}$, without affecting the statement of the main result (Theorem \ref{ThCinfty}) below, as
$$\| |b|^{\frac{1}{2}}(\lambda-\Delta)^{-\frac{\alpha-1}{4}}\|_{2 \rightarrow 2} \leq \sqrt{\delta}$$
for some $\lambda=\lambda_\delta>0$.

In examples 1-4, 6 below we list some sub-classes of $\mathbf{F}_{\delta}^{\scriptscriptstyle \frac{\alpha-1}{2}}$ defined in elementary terms.

\begin{examples}

1.~By the fractional Sobolev inequality, 
$$|b| \in L^{\frac{d}{\alpha-1}}+L^{\infty} \quad \Rightarrow \quad b \in \mathbf{F}_{\delta}^{\frac{\alpha-1}{2}},$$ where $\delta>0$ can be chosen arbitrarily small.

2.~More generally, vector fields  with entries in $L^{\frac{d}{\alpha-1},\infty}$ (the weak $L^{\frac{d}{\alpha-1}}$ class) are weakly form-bounded: 
\begin{align*}
b=b_1+b_2 & \in L^{\frac{d}{\alpha-1},\infty}(\mathbb R^d,\mathbb R^d) +  L^{\infty}(\mathbb R^d,\mathbb R^d) \\
&\Rightarrow \quad b \in \mathbf{F}_{\delta}^{\frac{\alpha-1}{2}}, \qquad
\sqrt{\delta}=\Omega_d^{-\frac{\alpha-1}{2d}}\frac{2^{-\frac{\alpha-1}{2}}\Gamma\big(\frac{d-\alpha+1}{4}\big)}{\Gamma\big(\frac{d+\alpha-1}{4}\big)}\|b_1\|_{\frac{d}{\alpha-1},\infty}^{\frac{1}{2}},
\end{align*}
where $\Omega_d$ is the volume of the unit ball $B(0,1) \subset \mathbb R^d$, see \cite[Corollary 2.9]{KPS}.

3.~In particular, by the fractional Hardy-Rellich inequality, the Hardy-type drift
$$
b(x)=\sqrt{\delta} \kappa_{\alpha,d}|x|^{-\alpha}x, \qquad \kappa_{\alpha,d}:=2^{\frac{\alpha-1}{2}-}\frac{\Gamma(\frac{d+\alpha-1}{4})}{\Gamma(\frac{d-\alpha+1}{4})}, \quad \delta>0
$$
is in $\mathbf{F}_\delta^{\scriptscriptstyle \frac{\alpha-1}{2}}$ with $\lambda=0$ \cite[Corollary 2.9]{KPS}.

4.~Recall that a vector field $b \in L^1_{\loc}(\mathbb R^d,\mathbb R^d)$ is said to belong to the Kato class $\mathbf{K}_{\delta}^{d,\alpha-1}$, $\delta>0$
if there exists $\lambda = \lambda_{\delta} > 0$ such that 
$$\big\| \big(\lambda+(-\Delta)^{\frac{\alpha}{2}}\big)^{-\frac{\alpha-1}{\alpha}}|b|\big\|_{\infty} \leq \delta.$$
We have $$\mathbf{K}_{\delta}^{d,\alpha-1} \subsetneq \mathbf{F}_{\delta}^{\frac{\alpha-1}{2}}.$$
Indeed, if $b \in \mathbf{K}_{\delta}^{d,\alpha-1}$, then by duality $\big\| |b|\big(\lambda+(-\Delta)^{\frac{\alpha}{2}}\big)^{-\frac{\alpha-1}{\alpha}}\big\|_{1 \rightarrow 1} \leq \delta$, and so by interpolation $\big\| |b|^{\frac{1}{2}}\big(\lambda+(-\Delta)^{\frac{\alpha}{2}}\big)^{-\frac{\alpha-1}{\alpha}}|b|^{\frac{1}{2}}\big\|_{2 \rightarrow 2} \leq \delta $, i.e.\,$b \in \mathbf{F}_{\delta}^{\frac{\alpha-1}{2}}$.

We note that for a given $\varepsilon>0$ there exist $b \in \bigcap_{\delta>0}\mathbf{K}_{\delta}^{d,\alpha-1}$ such that $|b| \not \in L_{\loc}^{1+\varepsilon}$.

It is not difficult to see that the vector field in example 3 does not belong to the Kato class $\mathbf{K}^{d,\alpha-1}_{\delta_1}$ for any $\delta_1>0$. In fact, even $L^{\frac{d}{\alpha-1}}(\mathbb R^d,\mathbb R^d) \not\subset \mathbf{K}^{d,\alpha-1}_{\delta_1}$ for any $\delta_1>0$. 

\smallskip

5.~We say that vector field $b$ belongs to the class of form-bounded vector fields $\mathbf{F}_\delta^{\alpha-1}$, $\delta>0$ if $|b| \in L^2_{\loc}$
$$
\big\||b|\big(\lambda+(-\Delta)^{\frac{\alpha}{2}}\big)^{-\frac{\alpha-1}{\alpha}}\big\|_{2 \rightarrow 2} \leq \delta \quad \text{ for some }\lambda=\lambda_\delta.
$$
By the Heinz-Kato inequality,  $\mathbf{F}_\delta^{\alpha-1} \subsetneq \mathbf{F}_\delta^{\scriptscriptstyle \frac{\alpha-1}{2}}$. (We note that $L^\frac{d}{\alpha-1}(\mathbb R^d,\mathbb R^d) \subset\mathbf{F}_\delta^{\alpha-1}$ with arbitrarily small $\delta$, however $\mathbf{K}^{d,\alpha-1}_{\delta_1} - \mathbf{F}_\delta^{\alpha-1} \neq \varnothing$ for any $\delta$, $\delta_1>0$.)

\smallskip

6.~If $|b|^{\frac{2}{\alpha-1}}$ belongs to the Campanato-Morrey class 
$$
\left\{v \in L_{\loc}^s \mid \biggl(\frac{1}{|Q|}\int_Q |v(x)|^s dx \biggr)^{\frac{1}{s}} \leq c_s l(Q)^{-2} \text{ for all cubes $Q$}\right\}, \quad s>1,
$$
where $|Q|$ and $l(Q)$ are the volume and the side length of a cube $Q$,
then $\||b|^{\frac{1}{\alpha-1}}(-\Delta)^{-\frac{1}{2}}\|_{2 \rightarrow 2} \leq \delta^{\frac{1}{\alpha-1}}$  with appropriate $\delta$ (Adams' inequality). Then, by the Heinz-Kato inequality,
$b \in \mathbf{F}_\delta^{\alpha-1}$ and so by the previous example $b \in \mathbf{F}_\delta^{\scriptscriptstyle \frac{\alpha-1}{2}}$.

More sophisticated examples of weakly form-bounded vector fields can be obtained by modifying examples in \cite[sect.\,3]{KiS}.

\end{examples}

Our point of departure is a simpler problem in $L^2$: to find the minimal assumption on $b$ such that $(-\Delta)^{\frac{\alpha}{2}} + b \cdot \nabla$ admits an operator realization on $L^2$ as the (minus) generator of a $C_0$ semigroup, say, $e^{-t\Lambda_2(b)}$. 
In Theorem \ref{thmL2} below we arrive at the condition $b \in \mathbf{F}_{\delta}^{\scriptscriptstyle \frac{\alpha-1}{2}}$, $\delta<1$. Theorem \ref{thmL2} first appeared in \cite{S} in the case $\alpha=2$. 

We note that applying to $(-\Delta)^{\frac{\alpha}{2}} + b \cdot \nabla$, $1<\alpha<2$ the form method, i.e.\,the Kato-Lions-Lax-Milgram-Nelson Theorem, is quite problematic since one can no longer employ the quadratic inequality in order to control the $b \cdot \nabla$ term. Moreover, even if $\alpha=2$, the form method can handle only the smaller class of vector fields $\mathbf{F}_\delta $ ($\equiv \mathbf{F}_\delta^1 \subsetneq \mathbf{F}_{\delta}^{\scriptscriptstyle \frac{1}{2}}$) while giving a weaker result on the regularity of the domain of $\Lambda_2(b)$ compared to \cite[Theorem 5.1]{S}, see detailed discussion in \cite{KiS}.
On the other hand, the Hille Perturbation Theorem, while applicable to $(-\Delta)^{\frac{\alpha}{2}} + b \cdot \nabla$ in $L^2$ for all $1<\alpha \leq 2$,
can handle only the proper sub-class $\mathbf{F}^{\alpha-1}_\delta$ of $\mathbf{F}_{\delta}^{\scriptscriptstyle \frac{\alpha-1}{2}}$, see \cite[Proposition 7]{KSS} for details. See also Remark 3 below.

Denote $C_\infty:=\{f \in C(\mathbb R^d): \lim_{|x| \rightarrow \infty}f(x)=0\}$ (with the $\sup$-norm). Recall that a positivity preserving contraction $C_0$ semigroup on $C_\infty$ is called  a Feller semigroup.

Now, having at hand an operator realization $\Lambda_2(b)$ of $(-\Delta)^{\frac{\alpha}{2}} + b \cdot \nabla$ in $L^2$, we are in position to enquire what extra assumption on $b \in \mathbf{F}_\delta^{\scriptscriptstyle \frac{\alpha-1}{2}}$, $\delta<1$ is needed to ensure that the operators $e^{-t\Lambda_2(b)} \upharpoonright L^2 \cap C_\infty$, $t>0$ admit extension to bounded linear operators on $C_\infty$ that constitute a Feller semigroup, say, $e^{-t\Lambda_{C_\infty}(b)}$.
In the main result of this paper, Theorem \ref{ThCinfty}, we show that this extra assumption is expressed in terms of the weak form-bound $\delta$: it has to be smaller than a certain explicit constant $c=c(d)<1$ (Theorem \ref{ThCinfty}(\textit{i}),(\textit{ii})).
The construction of the Feller semigroup proceeds via detailed regularity theory of $(-\Delta)^{\frac{\alpha}{2}} + b \cdot \nabla$ in $L^p$, $p>d-\alpha+1$ (Theorem \ref{thmLp}) which we develop, while imposing the same $L^2$ hypothesis on the drift (i.e.\,$b \in \mathbf{F}_{\delta}^{\scriptscriptstyle \frac{\alpha-1}{2}}$ but with smaller $\delta$), using the $L^p$ inequalities for symmetric Markov generators of \cite{BS,LS} (Appendix \ref{lp_sect}).

Let us note that the singularities of a  vector field $b \in \mathbf{F}_\delta^{\scriptscriptstyle\frac{\alpha-1}{2}}$ can be so strong that they destroy the standard bounds on the heat kernel $e^{-t\Lambda_{C_\infty}(b)}(x,y)$ in terms of $e^{-t(-\Delta)^{\scriptscriptstyle \frac{\alpha}{2}}}(x,y)$, see discussion below.

Next, in Proposition \ref{lem_weights}, we establish  weighted $L^p \rightarrow L^\infty$ estimates on the resolvent $(\mu+\Lambda_{C_\infty}(b))^{-1}$. In absence of the standard upper bound on the heat kernel $e^{-t\Lambda_{C_\infty}(b)}(x,y)$, these estimates play crucial role (e.g.\,they allow to prove that the Feller semigroup is conservative, i.e.\,$\int_{\mathbb R^d}e^{-t\Lambda_{C_\infty}(b)}(x,y)dy=1$ for all $x \in \mathbb R^d$).

Let $D([0,\infty[,\mathbb R^d)$ be the space of all right-continuous functions having left limits, endowed with the Skorokhod topology, $X_t$ the projection coordinate map on $D([0,\infty[,\mathbb R^d)$, and $\mathcal G_t$ is the filtration generated by $\{X_s, s \leq t\}$.
By a standard result, given a conservative Feller semigroup $T^t$ on $C_\infty$, there exist
probability measures $\{\mathbb P_x\}_{x \in \mathbb R^d}$ on $\mathcal G_\infty$ such that $\big(D([0,\infty[,\mathbb R^d), \mathcal G_t, \mathcal G_\infty, \mathbb P_x\big)$ is a Markov process,  $\mathbb P_{x}[X_0=x]=1$ and
$$
\mathbb E_{\mathbb P_x}[f(X_t)]=(T^t f)(x), \quad X \in D([0,\infty[,\mathbb R^d), \quad f \in C_\infty, \quad x \in \mathbb R^d.
$$ 

Finally, having at hand the weighted estimates, we run an $L^p$ weighted variant of an argument in \cite{PP,P} to show that, for every starting point $x \in \mathbb R^d$, the corresponding probability measure determined by $T^t:=e^{-\Lambda_{C_\infty}(b)}$ yields a weak solution to the SDE \eqref{eq0} (Theorem \ref{ThCinfty}(\textit{vi}),(\textit{vii})). 

The above program has been carried out in the case $\alpha=2$ for $b \in \mathbf{F}_\delta^{\scriptscriptstyle 1/2}$ in \cite{Kin, KiS} (Feller semigroup), \cite{KiS2} (the characterization of the probability measures as weak solutions to SDE \eqref{eq0} with Brownian motion in place of $Z_t$). The construction of the Feller semigroup in Theorem \ref{ThCinfty}(\textit{i}),(\textit{ii}) below follows closely \cite{Kin}, \cite[sect.\,4]{KiS}. The main novelty and difficulty is in the proof of the crucial weighted estimates of Proposition \ref{lem_weights} (Section \ref{weight_sect}).
The calculational techniques used in the proof of an analogous result in \cite{KiS2} are unavailable when $\alpha<2$. In this regard, we develop a new approach to the proof of these estimates
  taking advantage of the fact that the $L^p$ inequalities of \cite{BS,LS} are valid for abstract symmetric Markov generators, in particular, for a ``weighted'' fractional Laplace operator; we show that the latter is indeed a symmetric Markov generator using the method of proof of $L^1$ accretivity of non-local operators in weighted spaces introduced in \cite{KSS} (but for different weights and for different purpose).
Armed with the $L^p$ inequalities for the weighed fractional Laplace operator, we repeat the principal steps of construction of the Feller semigroup but now on the weighted space, using the fact that 
the crucial pointwise estimate \eqref{A} does not depend on the choice of the weight on $\mathbb R^d$.

In this paper, we prove a weaker uniqueness result than the uniqueness in law (i.e.\,we prove that the weak solution to \eqref{eq0}, determined by the Feller semigroup, is unique in the class of weak solutions that constitute an operator semigroup with reasonable properties, see Remark \ref{rem2} below). Concerning possible proof of the uniqueness in law, we note that in general $|\nabla u| \not \in L^\infty$, $u=(\mu+\Lambda_{C_\infty}(b))^{-1}f$, $b \in \mathbf{F}_\delta^{\scriptscriptstyle \frac{\alpha-1}{2}}$, even if $f \in C_c^\infty$. 

The method of this paper works for more general operators. In particular, in the construction of the Feller semigroup and in the proof of the weighted estimates below one can replace  $(-\Delta)^{\frac{\alpha}{2}}$ by a symmetric Markov generator $A$ provided  that it satisfies the pointwise estimate \eqref{A} (e.g.\,$A=(1-\Delta)^{\frac{\alpha}{2}}$).

\smallskip

Let us now comment more on the Kato class and on the existing results. 

Recall one of the equivalent definitions of the standard Kato class $\mathbf{K}_0^{d,\alpha-1}$:
$$
\mathbf{K}^{d,\alpha-1}_0:=\bigcap_{\delta>0}\mathbf{K}^{d,\alpha-1}_\delta,
$$ 
where $\mathbf{K}^{d,\alpha-1}_\delta$ has been defined in example 4 above. Since
$\mathbf{K}_\delta^{d,\alpha-1} \subsetneq \mathbf{F}_{\delta}^{\frac{\alpha-1}{2}}$, we have $$\mathbf{K}_0^{d,\alpha-1} \subsetneq \mathbf{F}_{\delta}^{\frac{\alpha-1}{2}} \quad \text{ for any fixed }\delta>0.$$ 
We note that $\mathbf{K}_0^{d,\alpha-1}$ is a proper sub-class of $\mathbf{K}^{d,\alpha-1}_\delta$, $\delta>0$. Concerning the difference between the two Kato classes, let us note that multiplying $b \in \mathbf{K}^{d,\alpha-1}_\delta$ by a constant $c>1$ in general takes the vector field out of  $\mathbf{K}^{d,\alpha-1}_\delta$ and thus out of $\mathbf{F}_{\delta}^{\frac{\alpha-1}{2}}$, while for $b \in \mathbf{K}_0^{d,\alpha-1}$ one has $cb \in \mathbf{K}_0^{d,\alpha-1}$ for arbitrarily large $c$.

It is seen, using H\"{o}lder's inequality, that $$|b| \in L^{p}+L^{\infty} \quad \Rightarrow \quad b \in \mathbf{K}_0^{d,\alpha-1}, \quad p > \frac{d}{\alpha -1}.$$

The Kato class $\mathbf{K}^{d,\alpha-1}_\delta$, with $\delta>0$ sufficiently small,
provides the standard bounds on heat kernel $e^{-t\Lambda(b)}(x,y)$, $\Lambda(b) = (-\Delta)^{\frac{\alpha}{2}} + b \cdot \nabla$:
\begin{equation}
\label{bds1}
C^{-1}e^{-t(-\Delta)^{\frac{\alpha}{2}}}(x,y) \leq e^{-t\Lambda(b)}(x,y) \leq Ce^{-t(-\Delta)^{\frac{\alpha}{2}}}(x,y), \quad x,y \in \mathbb R^d 
\end{equation}
for all $0<t<t_0$ for a constant $C=C(d,\alpha,b,t_0)>0$. Moreover, if $b \in \mathbf{K}^{d,\alpha-1}_0$, then $e^{-t\Lambda(b)}(x,y)$ is continuous. See \cite{BJ}.
The latter yields: $e^{-t\Lambda(b)}$ is a conservative Feller semigroup  in $C_u(\mathbb R^d)$ ($\equiv$ bounded uniformly continuous functions). It has been established in \cite{CKS} that the probability measures $\{\mathbb P_x\}_{x \in \mathbb R^d}$ determined by $e^{-t\Lambda(b)}$ solve the martingale problem for $(-\Delta)^{\frac{\alpha}{2}} + b \cdot \nabla$ with test functions in $C_c^\infty$, as needed to obtain two-sided bounds on the heat kernel of $X_t$ killed upon exiting a smooth bounded domain. The uniqueness in law of the weak solution to the martingale problem, as well as the existence and the uniqueness in law of the weak solution to SDE \eqref{eq0} with $b \in \mathbf{K}_0^{d,\alpha-1}$, were established later in \cite{CW}. In \cite{KS} the authors consider SDE \eqref{eq0} with a Kato class measure-valued drift and establish the corresponding heat kernel bounds. The case $\alpha=2$ was considered earlier in \cite{BC}.

\begin{remarks}
1.~Concerning the relationship between the Kato class condition and the Feller property, let us mention the following special case of a result in \cite{V}, \cite{OSSV}. Let $V \in L^1_{\loc}$ be of one sign (in fact, $V$ can be a measure). Under fairly general assumptions on $V$, one can construct an operator realization $H_1(V)$ of the fractional Schr\"{o}digner operator $(-\Delta)^{\frac{\alpha}{2}} + V$ in $L^1$ as (minus) generator of a $C_0$ semigroup. If $(e^{-tH_1(V)})^* \upharpoonright C_\infty \subset C_\infty$ and is a $C_0$ semigroup in $C_\infty$, then \textit{necessarily} $V$ is locally in the standard Kato class of potentials $\mathbf{K}^{d,\alpha}_{0}$ (i.e.\,for every compact $E \subset \mathbb R^d$, $\mathbf{1}_E V \in \mathbf{K}_0^d:=\cap_{\delta>0}\mathbf{K}^{d,\alpha}_\delta$, where $\mathbf{K}^{d,\alpha}_\delta:=\big\{W \in L^1_{\loc} \mid \|W\big(\lambda+(-\Delta)^{\frac{\alpha}{2}}\big)^{-1}\|_{1 \rightarrow 1} \leq \delta \text{ for some } \lambda=\lambda_\delta\big\}$).

The situation in the case of the fractional Kolmogorov operator $(-\Delta)^{\frac{\alpha}{2}} + b \cdot \nabla$ is different. Although whenever $b $ is in the standard Kato class of drifts $\mathbf{K}_{0}^{d,\alpha-1}$ this operator admits a realization $\Lambda_{C_\infty}(b)$ in $C_\infty$ as (minus) generator of a $C_0$ semigroup, by the result of this paper the class of admissible drifts for $\Lambda_{C_\infty}(b)$ can be enlarged to $\mathbf{F}_\delta^{\scriptscriptstyle \frac{\alpha-1}{2}}$ ($\supsetneq \mathbf{K}_{\delta}^{d,\alpha-1} \supsetneq \mathbf{K}_{0}^{d,\alpha-1}$) \textit{with positive} $\delta$.

\medskip

2.~Although the model vector field $b$ in example 3 above is so singular that it destroys the standard heat kernel bounds \eqref{bds1},
sharp heat kernel bounds on $e^{-t\Lambda(b)}(x,y)$ exist and depend explicitly on the weak form-bound $\delta$ via presence of a ``desingularizing'' weight $\varphi_t(y):=\varphi(t^{-\frac{1}{\alpha}}y)$
$$
C^{-1}e^{-t(-\Delta)^{\frac{\alpha}{2}}}(x,y)\varphi_t(y) \leq e^{-t\Lambda(b)}(x,y) \leq Ce^{-t(-\Delta)^{\frac{\alpha}{2}}}(x,y)\varphi_t(y), \quad x,y \in \mathbb R^d, \quad y \neq 0,
$$
for all $t>0$, where $\varphi \in C(\mathbb R^d - \{0\})$, $\varphi(y):=|y|^{-d+\beta}$ for appropriate $0<\beta<d$ if $|y|<1$, $\varphi(y):=\frac{1}{2}$ if $|y|>2$ \cite[Theorem 3]{KSS}. 

Theorem \ref{ThCinfty} below provides a probabilistic setting for \cite{KSS}.

\medskip

3.~The proof of Theorem \ref{thmL2} below ($L^2$ theory of $(-\Delta)^{\frac{\alpha}{2}} + b \cdot \nabla$, $b \in \mathbf{F}_\delta^{\scriptscriptstyle \frac{\alpha-1}{2}}$, $\delta<1$) appeals to  ideas of E.\,Hille and H.F.\,Trotter. Alternatively, one can use the approach appealing to ideas of E.\,Hille and J.-\,L.\,Lions based on considering the suitable chain of Hilbert spaces for $(-\Delta)^{\frac{\alpha}{2}} + b \cdot \nabla$. See details in \cite{KiS}. 

\medskip

4.~Consider operator $(-\Delta)^{\frac{\alpha}{2}} + b \cdot \nabla$ with $b$ in $\mathbf{F}_\delta^{\alpha-1}$, the class of form-bounded vector fields. In the case $\alpha=2$, \cite{KoS} constructed an operator realization of $\Delta - b \cdot \nabla$ as a Feller generator using a different approach. Despite the inclusion $\mathbf{F}_\delta \subset \mathbf{F}_\delta^{\scriptscriptstyle 1/2}$, the result in \cite{KoS} is not a special case of the result in \cite{Kin}, \cite[sect.\,4]{KiS} since it admits larger values of $\delta$. This alternative approach, however, is inapplicable in the case $\alpha<2$ (one can not use the quadratic inequality in order to control the $b \cdot \nabla$ term).
\end{remarks}

\medskip
\medskip

\noindent\textbf{Notation.}\;Let $\mathcal W^{s,p}$, $s>0$ be the Bessel potential space endowed with norm $\|u\|_{p,s}:=\|g\|_p$,  
$u=(1+(-\Delta)^{\frac{\alpha}{2}})^{-\frac{s}{\alpha}}g$, $g \in L^p$, and $\mathcal W^{-s,p'}$ the anti-dual of $\mathcal W^{s,p}$, $p'=\frac{p}{p-1}$.

By $\mathcal B(X,Y)$ we denote the space of bounded linear operators between Banach spaces $X \rightarrow Y$, endowed with the operator norm $\|\cdot\|_{X \rightarrow Y}$. Abbreviate $\mathcal B(X):=\mathcal B(X,X)$. 

We write $T=s\mbox{-} X \mbox{-}\lim_n T_n$ for $T$, $T_n \in \mathcal B(X)$, $n=1,2,\dots$, if $Tf=\lim_n T_nf$ in $X$ for every $f \in X$.

Let $\mathcal L^d$ be the standard Lebesgue measure on $\mathbb R^d$. Denote
$$
\langle h\rangle:=\int_{\mathbb R^d} hd\mathcal L^d, \quad \langle h,g\rangle:=\langle h\bar{g}\rangle.
$$

Let $$
\gamma (x)=\left \{ 
\begin{array}{rl}
c e^{-\frac{1}{|x|^{2}-1}} & \text{ if } |x| < 1, \\
0 \hspace*{0.5cm} & \text{ if } |x| \geq 1, 
\end{array} 
\right.
$$ 
where $c$ is adjusted to    $\langle \gamma \rangle =1$.
Define the standard mollifier  $$\gamma_{\varepsilon}(x):=\frac{1}{\varepsilon^{d}}\gamma\bigg(\frac{x}{\varepsilon}\bigg), \quad x \in \mathbb{R}^{d}, \quad \varepsilon>0.$$

Given a vector field $b \in \mathbf{F}_{\delta}^{\scriptscriptstyle \frac{\alpha-1}{2}}$, we fix its $C^\infty$ smooth approximation
\begin{equation*}
b_{n}:=\gamma_{\varepsilon_{n}} \ast (\mathbf{1}_{n}b), \quad \varepsilon_{n}\downarrow 0, \quad n=1,2,\dots,
\end{equation*}
where $\mathbf{1}_{n}$ is the indicator of  $\{x \in \mathbb{R}^{d} \mid |x| \leq n, |b(x)| \leq n \}$.

It is seen that for every $\tilde{\delta} > \delta$ one can select $\varepsilon_{n}\downarrow 0$ so that  $b_{n}  \in \mathbf{F}_{\tilde{\delta}}^{\scriptscriptstyle \frac{\alpha-1}{2}}$ with $\lambda \neq \lambda(n)$.
Our assumptions concerning $\delta$ below are strict inequalities, so we can assume without loss of generality that $\tilde{\delta}=\delta$.

\medskip

\noindent\textbf{Main results.}\;Set $A: = (-\Delta)^{\frac{\alpha}{2}}$. Define constant $m_{d,\alpha}$ by the pointwise inequality
\begin{equation}
\label{A}
\tag{A.0}
\big|\nabla_y \big(\mu+A\big)^{-1}(x,y)| \leq m_{d,\alpha}\big(\kappa^{-1}\mu+A\big)^{-\frac{\alpha-1}{\alpha}}(x,y)
\end{equation}
for all $x,y \in \mathbb R^d$, $x \neq y$, $\mu>0$ for some $\kappa = \kappa_{d,\alpha}>0$
(for a simple estimate on $m_{d,\alpha}$ from above, see  Appendix \ref{app_est}, the proof of \eqref{estimation2}).

\begin{theorem}
\label{ThCinfty}
Let $d \geq 3$, $b \in \mathbf{F}_\delta^{\frac{\alpha-1}{2}}$ with  $\delta<m_{d,\alpha}^{-1}4\bigl[\frac{d-\alpha}{(d-\alpha+1)^2} \wedge \frac{\alpha(d+\alpha)}{(d+2\alpha)^2}\bigr]$. 
The following is true.

{\rm (\textit{i})} Set $\Lambda_{C_{\infty}}(b_{n})=(-\Delta)^{\frac{\alpha}{2}}+b_{n} \cdot\nabla$, $D(\Lambda_{C_{\infty}}(b_{n}))=\big(1+(-\Delta)^{\frac{\alpha}{2}}\big)^{-1}C_\infty$. The limit 
\begin{equation*}
s{\mbox-}C_{\infty}\mbox{-}\lim_{n}e^{-t\Lambda_{C_{\infty}}(b_{n})} \quad \text{{\rm(}loc.\,uniformly in $t \geq 0${\rm)}}
\end{equation*}
exists and determines a Feller semigroup $T^t=:e^{-t\Lambda_{C_{\infty}}(b)}$ whose generator $\Lambda_{C_\infty}$ is an appropriate operator realization of the formal operator $-(-\Delta)^{\frac{\alpha}{2}} - b \cdot \nabla$ in $C_\infty$.

\smallskip

{\rm (\textit{ii})} There exist $\mu_0>0$ and $p>d-\alpha+1$ such that, for all $\mu \geq \mu_0$,
$$\big[(\mu+\Lambda_{C_{\infty}}(b))^{-1} \upharpoonright C_\infty \cap L^{p}\big]^{\rm clos}_{L^p \rightarrow C_\infty} \in \mathcal{B}(L^{p},C_\infty),$$
and
$$\big(\mu+\Lambda_{C_\infty}(b)\big)^{-1}[C_\infty \cap L^p] \subset C^{0,\gamma}, \quad \gamma < 1-\frac{d-\alpha+1}{p}.$$

$e^{-t\Lambda_{C_\infty}(b)}$, $t>0$ is an integral operator.

\smallskip

{\rm (\textit{iii})} For every $p \in [2,p_+[$, $p_+=\frac{2}{1 - \sqrt{1-m_{d,\alpha}\delta}}$, and all $1<r<p<q<\infty$,
$$
(\mu + \Lambda_{C_\infty}(b))^{-1} \upharpoonright C_\infty \cap L^p \text{ extends by continuity to $\mathcal B\big(\mathcal W^{-\frac{\alpha-1}{r'},p},\mathcal W^{1+\frac{\alpha-1}{q},p}\big)$}.
$$

{\rm (\textit{iv})} $
\big(\mu+\Lambda_{C_\infty}(b)\big)^{-1} \upharpoonright C_\infty \cap L^2 \text{ extends by continuity to $\mathcal B(\mathcal W^{-\frac{\alpha-1}{2},2},\mathcal W^{\frac{\alpha+1}{2},2})$}.
$

\smallskip

{\rm (\textit{v})} $e^{-t\Lambda_{C_{\infty}}(b)}$ is conservative, i.e.\,$\int_{\mathbb R^d }e^{-t\Lambda_{C_{\infty}}(b)}(x,y)dy=1$\; $\forall\,x \in \mathbb R^d$.

\smallskip

Let $\{\mathbb{P}_x\}_{x \in \mathbb R^d}$ be the probability measures determined by $e^{-t\Lambda_{C_{\infty}}(b)}$.

\smallskip

{\rm (\textit{vi})} For every $x \in \mathbb R^d$ and $t>0$, $\mathbb E_{\mathbb P_x}\int_0^t |b(X_s)|ds<\infty$.

\smallskip

{\rm (\textit{vii})} There exists a process $Z_t$ with trajectories in $D([0,\infty[,\mathbb R^d)$, which is a symmetric $\alpha$-stable process under each $\mathbb P_x$, such that
$$
X_t=x-\int_0^t b(X_s)ds+Z_t, \quad t \geq 0.
$$

\end{theorem}

\begin{remark}
The domain of the Feller generator $D\big(\Lambda_{C_\infty}(b)\big)$ does not admit  description in elementary terms.
In particular, even for $b \in L^\infty(\mathbb R^d,\mathbb R^d) - C(\mathbb R^d,\mathbb R^d)$, $C_c^\infty \not\subset D\big(\Lambda_{C_\infty}(b)\big)$. 
\end{remark}

\begin{remark}
In Theorem \ref{thmLp}(\textit{vi}) below we show that $u=(\mu + \Lambda_{C_\infty}(b))^{-1}f$, $f \in C_\infty \cap L^p$ is a weak solution to the corresponding elliptic equation in $L^p$. 
\end{remark}

\begin{remark}
Having at hand Theorem \ref{ThCinfty}(\textit{i}) and the estimates of Proposition \ref{lem_weights} below, one can show repeating the argument in \cite{KiS2} that, for every $f \in C_c^\infty$, the process
$$
t \mapsto f(X_t)-f(x)+\int_0^t \bigl[ (-\Delta)^{\frac{\alpha}{2}}f(X_s) + (b \cdot \nabla f)(X_s)\bigr]ds, \quad t \geq 0, \quad X \in D([0,\infty[,\mathbb R^d), 
$$
is a $\mathbb P_x$-martingale. 
Repeating the corresponding argument in \cite{CKS} one can further show that $X_t$ has the same L\'{e}vy system as symmetric $\alpha$-stable process.
\end{remark}

\begin{remark}
\label{rem2}
It is not difficult to prove that the Feller property and property (\textit{iv}) determine $\{\mathbb P_x\}_{x \in \mathbb R^d}$ uniquely.
Precisely, suppose that for every $x \in \mathbb R^d$ we are given a weak solution $\mathbb Q_x$ to SDE \eqref{eq0}.
Define for every $f \in C_c^\infty$
$$
R^Q_\mu f(x):=\mathbb E_{\mathbb Q_x} \int_0^\infty e^{-\mu s} f(X_s)ds, \quad X_s \in D([0,\infty[,\mathbb R^d), \quad x \in \mathbb R^d, \quad \mu>\lambda_\delta.
$$ 
In Appendix \ref{u_app} we show that if $R^Q_\mu C_c^\infty \subset C_b$ and $R^Q_\mu \upharpoonright C_c^\infty$ admits extension by continuity to $\mathcal B(\mathcal W^{-\frac{\alpha-1}{2},2}, L^2)$, then $R_\mu^Qf=(\mu+\Lambda_{C_\infty}(b))^{-1}f$ ($f \in C_c^\infty$), and so $\{\mathbb Q_x\}_{x \in \mathbb R^d}=\{\mathbb P_x\}_{x \in \mathbb R^d}$. 
Alternatively, in the assumptions of Theorem \ref{ThCinfty}, one can repeat the proof of the uniqueness result in \cite{KiS2}: if $\{\mathbb Q_x\}_{x \in \mathbb R^d}$ are weak solutions to  SDE \eqref{eq0} obtained via a `reasonable' approximation procedure, i.e.\,
$$\mathbb Q_x=w{\mbox-}\lim_n \mathbb P_x(\tilde{b}_n), \quad x \in \mathbb R^d$$
such that the smooth vector fields $\tilde{b}_n$ are weakly form-bounded with the same weak form-bound $\delta$ (and $\lambda \neq \lambda(n)$), then $\{\mathbb Q_x\}_{x \in \mathbb R^d}=\{\mathbb P_x\}_{x \in \mathbb R^d}$. 
\end{remark}

In absence of the upper bound on the heat kernel $e^{-t\Lambda_{C_\infty}(b)}(x,y)$, the following weighted estimates play crucial role in the proof of Theorem \ref{ThCinfty}.
Set
$$\eta(x):=(1+|x|^2)^{\nu}, \quad 0<\nu<\frac{\alpha}{2}.$$
Denote $L^p_\eta:= L^p(\mathbb R^d, \eta^2 d\mathcal L^d)$, $\|\cdot\|_{p,\eta}^p:=\langle |\cdot|^p \eta ^2\rangle $. 

\begin{proposition}
\label{lem_weights}
\label{weight_lem}
Let $d \geq 3$, $b \in \mathbf{F}_\delta^{\frac{\alpha-1}{2}}$ with $\delta<m_{d,\alpha}^{-1}4\bigl[\frac{d-\alpha}{(d-\alpha+1)^2} \wedge \frac{\alpha(d+\alpha)}{(d+2\alpha)^2}\bigr]$. There exist $0<\nu<\alpha/2$, $p > (d-\alpha+1) \vee (\frac{d}{2\nu}+2)$ and $\mu_0>0$ such that
for every $h \in C_c$, $\mu \geq \mu_0$
\begin{align}
\label{j_1_w}
\tag{$E_1$}
\|\eta^{-1} (\mu+\Lambda_{C_\infty}(b))^{-1} \eta h\|_\infty & \leq K_1\|h\|_{p,\eta}, \\
\label{j_2_w}
\tag{$E_2$}
\|\eta^{-1} (\mu+\Lambda_{C_\infty}(b))^{-1}\eta |b_m| h\|_\infty & \leq K_2\| |b_m|^\frac{1}{p}  h\|_{p,\eta},  \\
\label{j_3_w}
\tag{$E_3$}
\|\eta^{-1} |b_m|^\frac{1}{p}(\mu+\Lambda_{C_\infty}(b))^{-1}\eta |b_m| h\|_{p,\eta} & \leq K_3 \| |b_m|^\frac{1}{p}  h\|_{p,\eta}, 
\end{align}
where $K_i>0$, $i=1,2,3$, do not depend on $m=1,2,\dots$ The constant $K_3$ can be chosen arbitrarily small at expense of increasing $\mu_0$.
\end{proposition}

\noindent\textbf{Acknowledgements.} We would like to express our gratitude to Yu.\,A.\,Sem\"{e}nov for fruitful discussions.

\tableofcontents

\section{Proof of Theorem \ref{ThCinfty}(\textit{i}), (\textit{ii})}

\textbf{1.~}Set $A := (-\Delta)^{\frac{\alpha}{2}}$.
Let  $b \in \mathbf{F}_\delta^{\frac{\alpha-1}{2}}$, $\delta>0$.
Define
$$H:=|b|^{\frac{1}{2}}(\bar{\zeta}+A)^{-\frac{\alpha-1}{2\alpha}}, \quad S:=b^{\frac{1}{2}} \cdot \nabla (\zeta+A)^{-\frac{\alpha+1}{2\alpha}}, \quad b^{\frac{1}{2}}:=b|b|^{-\frac{1}{2}}, \quad\Real\,\zeta \geq \lambda.$$ 

\begin{lemma}
$H$, $S \in \mathcal B(L^2)$ and $\|H^*S\|_{2 \rightarrow 2} \leq \delta.$
\end{lemma}
\begin{proof}
Indeed,
$\|H^*S\|_{2 \rightarrow 2} \leq  \|(\zeta+A)^{-\frac{\alpha-1}{2\alpha}}|b|^{\frac{1}{2}}\|_{2 \rightarrow 2} \| b^{\frac{1}{2}} \cdot \nabla (\zeta+A)^{-\frac{\alpha+1}{2\alpha}}\|_{2 \rightarrow 2},$
and
$$\|(\zeta+A)^{-\frac{\alpha-1}{2\alpha}}|b|^{\frac{1}{2}}\|_{2 \rightarrow 2} \leq \|(\Real\zeta+A)^{-\frac{\alpha-1}{2\alpha}}|b|^{\frac{1}{2}}\|_{2 \rightarrow 2} \leq \sqrt{\delta} \quad \text{ by } b \in \mathbf{F}_\delta^{\scriptsize \frac{\alpha-1}{2}},$$
 $$\|b^{\frac{1}{2}} \cdot \nabla (\zeta+A)^{-\frac{\alpha+1}{2\alpha}}\|_{2 \rightarrow 2} \leq \||b|^{\frac{1}{2}}  (\zeta +A)^{-\frac{\alpha-1}{2\alpha}}\|_{2 \rightarrow 2}\|\nabla (\zeta+A)^{-\frac{1}{\alpha}}\|_{2 \rightarrow 2} \leq \sqrt{\delta},$$
where we used $\|\nabla (\zeta+A)^{-\frac{1}{\alpha}}\|_{2 \rightarrow 2} \leq 1$ (for $\|\nabla g \|_{2}=\|A^{\frac{1}{\alpha}}g\|_{2}$ and $\|A^{\frac{1}{\alpha}}(\zeta +A)^{-\frac{1}{\alpha}}\|_{2 \rightarrow 2} \leq 1$ by the Spectral Theorem).
\end{proof}

Let $\delta<1$. Define operator-valued function
\begin{align*}
\Theta_2 (\zeta,b)& :=(\zeta +A )^{-\frac{\alpha+1}{2\alpha}}(1+H^*S)^{-1}(\zeta +A )^{-\frac{\alpha-1}{2\alpha}}  \\
&=(\zeta +A )^{-1}-(\zeta+A)^{-\frac{\alpha+1}{2\alpha}}H^{*}(1+SH^{*})^{-1}S(\zeta +A )^{-\frac{\alpha-1}{2\alpha}}
\in \mathcal B(L^2).
\end{align*}

\begin{theorem}[$L^2$ theory]
\label{thmL2}
Let  $b \in \mathbf{F}_\delta^{\frac{\alpha-1}{2}}$, $\delta<1$. 
The following is true:

\smallskip

(\textit{i}) There exists a closed densely defined operator $\Lambda_2(b)$ on $L^2$ such that
$$\Theta_2 (\zeta,b)=\bigl(\zeta + \Lambda_{2}(b)\bigr)^{-1},\quad \Real\zeta \geq \lambda.$$
$-\Lambda_2(b)$ is the generator of a quasi bounded holomorphic semigroup.

\medskip

{\rm(\textit{ii})} $(\mu+\Lambda_2(b))^{-1}$ extends by continuity to $\mathcal B(\mathcal W^{-\frac{\alpha-1}{2},2},\mathcal W^{\frac{\alpha+1}{2},2})$, $\mu \geq \lambda$.

\medskip

{\rm(\textit{iii})}
 $e^{-t\Lambda_{2}(b_{n})}\stackrel{\rm s}{\rightarrow} e^{-t\Lambda_{2}(b)}$ in $L^2$ locally uniformly in $t \geq 0$,

\smallskip

\noindent where $\Lambda_{2}(b_{n}):=A+b_{n} \cdot\nabla$, $D(\Lambda_{2}(b_{n}))=(1+A)^{-1}L^2$.
\end{theorem}

\begin{proof}
The proof follows closely the proof of \cite[Theorem 5.1]{S}, \cite[Theorem 4.4]{KiS} (there $\alpha=2$), and goes in several steps:

\begin{center}
$1^\circ$.~$(\zeta + \Lambda_{2}(b_{n}))^{-1}=\Theta_2 (\zeta , b_{n})$ for $\Real\zeta > c_n$, $c_n\uparrow \infty$, $n=1,2,\dots$

\medskip

$2^\circ$.~$\Theta (\zeta,b_{n})$ is a pseudo-resolvent on $\{\Real \zeta \geq \lambda\}$. 

\medskip

$3^\circ$.~$(\zeta + \Lambda_{2}(b_{n}))^{-1}=\Theta (\zeta , b_{n})$ for all $\Real \zeta \geq \lambda$.

\medskip

$4^\circ$.~$\|\Theta (\zeta,b_{n}) \|_{2 \rightarrow 2}  \leq (1-\delta)^{-1} |\zeta|^{-1}$, $\Real \zeta \geq \lambda$.

\medskip

$5^\circ$.~$\mu\Theta (\mu,b_{n})\stackrel{\rm s}{\rightarrow} 1$ in $L^{2}$ as $\mu \uparrow \infty$ uniformly in $n$.

\medskip

$6^\circ$.~$\Theta (\zeta,b_{n})\stackrel{\rm s}{\rightarrow}\Theta (\zeta,b)$ in $L^2$ for every $\Real\zeta \geq \lambda$.
\end{center}

Steps $3^\circ$-$6^\circ$ verify conditions of the 
Trotter Approximation Theorem (Appendix \ref{app_trotter}) $\Rightarrow$  Theorem \ref{thmL2}(\textit{i}), (\textit{iii}). (\textit{ii}) is immediate from (\textit{i}) and the definition of $\Theta(\mu,b)$.

\smallskip

Let us comment on the proof of $1^\circ$-$6^\circ$, referring to \cite[sect.\,4.2]{KiS} for details.

Proof of $1^\circ$. It is clear that $\|b_{n} \cdot \nabla (\zeta +A)^{-1}\|_{2 \rightarrow 2} \leq n \|(\zeta +A)^{-\frac{\alpha-1}{\alpha}}\|_{2 \rightarrow 2} \leq \frac{1}{2}$ for $0<\Real \zeta$ sufficiently large, so by the Hille Perturbation Theorem (see e.g.\,\cite[Ch.\,IX, sect.\,2.2]{Ka}), the algebraic sum $-\Lambda_{2}(b_{n}):=-(A+b_{n}\cdot\nabla)$, $D(\Lambda_{2}(b_{n}))=(1+A)^{-1}L^2$ generates holomorphic $C_0$ semigroup.
Comparing the Neumann series for $(\zeta + \Lambda_2(b_n))^{-1}$ with $\Theta_2(\eta,b_n)$, we obtain $1^\circ$.

The pseudo-resolvent identity $\Theta (\zeta,b_{n})-\Theta (\eta,b_{n})=(\eta - \zeta )\Theta (\zeta,b_{n})\Theta (\eta,b_{n})$, $\Real \zeta, \Real \eta \geq \lambda$ follows by direct calculations $\Rightarrow$ $2^\circ$.

Proof of $3^\circ$. By $2^\circ$, the null set and the range of $\Theta (\zeta,b_{n})$ do not depend on $\zeta$. By $1^\circ$, the common null set of $\Theta (\zeta,b_{n})$ is $\{0\}$, and the common range is dense in $L^2$.
Thus, by a theorem of E.\,Hille \cite[Sect.\,5.2]{HP},  \cite[Ch.\,VIII, sect.\,4]{Yos}, $\Theta (\zeta,b_{n})$, $\Real \zeta \geq \lambda$, is the resolvent of a densely defined operator which, by 1, must coincide with $\Lambda_2(b_n)$. 

Proof of $4^\circ$, $5^\circ$ follows from the definition of $\Theta_2(\zeta,b)$.

Proof of $6^\circ$ follows from the definition of $\Theta_2(\zeta,b)$ using the Dominated Convergence Theorem.
\end{proof}

\begin{remark*}
The semigroup $e^{-t\Lambda_2(b)}$ is only quasi bounded, so the fact that it is holomorphic is an indispensable element of the construction.
\end{remark*}

\textbf{2.~}Since $e^{-t\Lambda_{2}(b)}$ is a $L^\infty$ contraction (e.g.\,by Theorem \ref{thmL2}(\textit{ii})), 
one obtains (by interpolation) a consistent family 
of quasi bounded semigroups on $L^p$, $p \in [2,\infty[$, defined by
\begin{equation*}
e^{-t\Lambda_p(b)}:=\left[ e^{-t\Lambda_{2}(b)}\upharpoonright L^2 \cap C_{\infty}\right]_{p \rightarrow p}^{\rm clos} \in \mathcal{B}(L^{p}).
\end{equation*}
The generator $-\Lambda_{p}(b)$ is an appropriate operator realization of $-A-b \cdot \nabla$ in $L^p$.

Let $\delta$ satisfy $m_{d,\alpha}\delta<1$. Denote
$$
1<p_{\pm}:=\frac{2}{1 \mp \sqrt{1-m_{d,\alpha}\delta}}<\infty.
$$
Recall  $A=(-\Delta)^{\frac{\alpha}{2}}$.

\begin{theorem}[$L^p$ theory]
\label{thmLp}
Let $b \in \mathbf{F}_\delta^{\frac{\alpha-1}{2}}$, $m_{d,\alpha}\delta<1$. The following is true:

\smallskip

{\rm(\textit{i})} For every $p \in[2,p_+[$ the resolvent set of $-\Lambda_p(b)$ contains $\{\mu \geq \kappa\lambda\}$, and
 $$(\mu + \Lambda_{p}(b))^{-1}=\Theta_{p} (\mu,b),$$
where
\begin{align*} 
\Theta_{p} (\mu,b) & := (\mu +A)^{-1} -(\mu +A)^{-1/\alpha +(-1+\frac{1}{\alpha})/q}Q_{p}(q)(1+T_{p})^{-1}G_{p}(r)(\mu +A)^{ (-1+\frac{1}{\alpha})/r'} \in \mathcal B(L^p), 
\end{align*}
$$
G_{p}(r) := b^{\frac{1}{p}}\cdot \nabla (\mu + A)^{-1/\alpha +(-1+\frac{1}{\alpha})/r} \in \mathcal B(L^p), \quad r<p,
$$
$$
T_{p} :=b^{\frac{1}{p}}\cdot \nabla(\mu + A)^{-1}|b|^{\frac{1}{p'}} \text{ on } \mathcal{E}:=\cup_{\varepsilon>0}e^{-\varepsilon|b|}L^{p}, 
$$
$$
Q_{p}(q) := (\mu + A)^{(-1+\frac{1}{\alpha})/q'}|b|^{\frac{1}{p'}} \text{ on } \mathcal{E}, \quad q>p,
$$
operators $T_p$, $Q_p(q)$
admit extensions by continuity to $\mathcal B(L^p)$, which we denote again by $Q_p(q)$ and $T_p$; 
$$
\|T_p\|_{p \rightarrow p} \leq m_{d,\alpha}c_p\delta<1, \quad c_p:=\frac{pp'}{4},
$$ 
$$
\|G_p(r)\|_{p \rightarrow p} \leq M_{1,r}, \quad \|Q_p(q)\|_{p \rightarrow p} \leq M_{2,q}, 
$$
where constants $M_{1,q} \neq M_{1,q}(\mu)$, $M_{2,r} \neq M_{2,r}(\mu)$.

\smallskip

{\rm(\textit{ii})} From the definition of $\Theta_p(\mu,b)$, 
$$
(\mu + \Lambda_{p}(b))^{-1} \text{ extends by continuity to $\mathcal B\big(\mathcal W^{-\frac{\alpha-1}{r'},p},\mathcal W^{1+\frac{\alpha-1}{q},p}\big)$},
$$
$$D(\Lambda_{p}(b)) \subset \mathcal{W}^{1+\frac{\alpha -1}{q},p}, \quad q > p.$$ 
In particular, if $m_{d,\alpha}\delta < 4\frac{d-\alpha}{(d-\alpha+1)^{2}}$, there exists   $p \in ]d-\alpha+1,p_+[$ such that, by the Sobolev Embedding Theorem, $$D(\Lambda_{p}(b)) \subset C^{0,\gamma}, \quad \gamma < 1-\frac{d-\alpha+1}{p}.$$

\smallskip

{\rm(\textit{iii})}
$e^{-t\Lambda_{p}(b_{n})} \rightarrow e^{-t\Lambda_{p}(b)}$ in $L^p$ locally uniformly in $t \geq 0$,

\smallskip

\noindent where $\Lambda_{p}(b_{n}):=A+b_{n} \cdot\nabla$, $D(\Lambda_{p}(b_{n}))=(1+A)^{-1}L^p$.

\smallskip

{\rm(\textit{iv})}
$
\|e^{-t\Lambda_{r}(b)}\|_{r \rightarrow q} \leq c_r e^{\omega t}t^{-\frac{d}{\alpha}(\frac{1}{r}-\frac{1}{q})}$, $2 \leq r < q \leq \infty$, $\omega:=\frac{2\lambda}{r}.
$

\smallskip

{\rm(\textit{v})} $e^{-t\Lambda_{p}(b)}$, $t>0$ are integral operators.

\smallskip

{\rm(\textit{vi})} 
$ \langle \Lambda_p (b) u, v \rangle = \langle u, (-\Delta)^{\frac{\alpha}{2}} v \rangle + \langle b\cdot\nabla u, v \rangle$, $u \in D(\Lambda_p(b))$, $v \in C_c^\infty(\mathbb{R}^d)$.
\end{theorem}

\begin{proof} The proof follows closely the proof of \cite[Theorem 4.4]{KiS}. 

(\textit{i}) We will use crucially the estimates (\textbf{a})-(\textbf{c}) of Lemma \ref{abc_lem} in Appendix \ref{lp_sect} (there $V:=|b|$). For all $f \in \mathcal E$
\begin{align*}
\|T_{p}f \|_{p} & =\|b^{\frac{1}{p}}\cdot \nabla (\mu +A)^{-1}|b|^{\frac{1}{p'}}f\|_{p} \\
& (\text{we are using \eqref{A}}) \\
&\leq m_{d,\alpha} \||b|^{\frac{1}{p}}(\kappa^{-1}\mu +A)^{-\frac{\alpha-1}{\alpha}}|b|^{\frac{1}{p'}}f\|_{p} \\ 
&(\text{we are using (\textbf{b})}) \\
&\leq  m_{d,\alpha} c_{p}\delta\|f\|_{p}, 
\end{align*}
where $m_{d,\alpha} c_{p}\delta<1$ since $p \in ]p_-,p_+[$.

In order to estimate $\|G_{p}(r)|_{p \rightarrow  p}$ and $\|Q_{p}(q)\|_{p \rightarrow  p}$, we will need the formula
\begin{equation}
\label{eq}
\tag{$\ast\ast$}
(\mu+A)^{-\tau}=\frac{\sin \pi \tau}{\pi}\int_{0}^{\infty}t^{-\tau}(t+\mu+A)^{-1}dt, \quad 0 < \tau < 1.
\end{equation}
Let  $f \in \mathcal{E}$ and  $\mu \geq \lambda$. We have for $q>p$
\begin{align*}
\|Q_{p}(q)f\|_{p} &=  \|(\mu +A)^{(-1+\frac{1}{\alpha})/q'}|b|^{\frac{1}{p'}}f\|_{p}\\
&(\text{we are using \eqref{eq}})  \\
&\leq   k_{\alpha,q} \int_{0}^{\infty} t^{(-1+\frac{1}{\alpha})/q'}\|(t+\mu +A)^{-1}|b|^{\frac{1}{p'}}f\|_{p}dt \\
&\leq   k_{\alpha,q}\int_{0}^{\infty}t^{(-1+\frac{1}{\alpha})/q'}(t+\mu)^{-\frac{1}{\alpha}}\|(t+\mu +A)^{-1+\frac{1}{\alpha}}|b|^{\frac{1}{p'}}f\|_{p}dt \\
&(\text{we are using (\textbf{c})}) \\
&\leq  k_{\alpha,q}(c_{p}\delta)^{1/p}\left(\int_{0}^{\infty}t^{(-1+\frac{1}{\alpha})/q'}(t+\mu)^{-\frac{1}{\alpha}+(-1+\frac{1}{\alpha})/p}dt\right) \|f\|_{p} \\
&= M_{2,q} \|f\|_{p},
\end{align*}
where, clearly  $M_{2,q}  < \infty$ because $q > p$.

For every $\mu \geq \lambda$ and  $r < p$ we have
\begin{align*}
\|G_{p}(r)f\|_{p}&=\|b^{\frac{1}{p}}\cdot \nabla (\mu + A)^{-1/\alpha +(-1+\frac{1}{\alpha})/r}f\|_{p}\\
&(\text{we are using \eqref{estimation2} in Appendix \ref{app_est}})\\
& \leq  c_{d,\gamma}\||b|^{\frac{1}{p}}(\mu + A)^{(-1+\frac{1}{\alpha})/r}f\|_{p}\\
& (\text{we are using  \eqref{eq}})\\
&\leq   c_{d,\gamma}k_{\alpha, r}\int_{0}^{\infty}t^{(-1+\frac{1}{\alpha})/r}\||b|^{\frac{1}{p}}(t+\mu + A)^{-1}|f|\|_{p}dt\\
& \leq  c_{d,\gamma}k_{\alpha, r}\int_{0}^{\infty}t^{(-1+\frac{1}{\alpha})/r}\||b|^{\frac{1}{p}}(t+\mu + A)^{-1+\frac{1}{\alpha}}\|_{p \rightarrow p} \|(t+\mu +A)^{-\frac{1}{\alpha}}|f|\|_{p} dt \\
&(\text{we are using (\textbf{a})})\\
& \leq  c_{d,\gamma}k_{\alpha, r}(c_{p}\delta)^{1/p}\left(\int_{0}^{\infty}t^{(-1+\frac{1}{\alpha})/r}(t+\mu)^{-\frac{1}{\alpha}+(-1+\frac{1}{\alpha})/p'}dt\right) \|f\|_{p}\\
&=  M_{1,r} \|f\|_{p},
\end{align*}
where $M_{1,r} < \infty$ because $r < p$.

Thus, $\Theta_p(\mu,b)$ is well defined.
Now, we have
\begin{equation*}
\Theta_2(\mu,b) \upharpoonright L^2 \cap L^p=\Theta_p(\mu,b) \upharpoonright L^2 \cap L^p, \quad \mu \geq \kappa\lambda, \quad p \in ]p_-,p_+[.
\end{equation*}
By the construction of $e^{-t\Lambda_p(b)}$, the latter yields (\textit{i}).

\smallskip

Clearly, (\textit{i}) $\Rightarrow$ (\textit{ii}).

\smallskip

(\textit{iii}) For every $f \in L^2 \cap L^\infty$, $$\|e^{-t\Lambda_{p}(b)}f-e^{-t\Lambda_{p}(b_{n})}f\|_p^p \leq \|e^{-t\Lambda_{p}(b)}f-e^{-t\Lambda_{p}(b_{n})}f\|_\infty^{p-2}\|e^{-t\Lambda_{p}(b)}f-e^{-t\Lambda_{p}(b_{n})}f\|_2^2,$$ and so the convergence follows from $e^{-t\Lambda_p(b)} \upharpoonright L^2 \cap L^p = e^{-t\Lambda_2(b)} \upharpoonright L^2 \cap L^p$, the $L^\infty$ contractivity of $e^{-t\Lambda_2(b)}$, $e^{-t\Lambda_2(b_n)}$, and the $L^2$ convergence of Theorem \ref{thmL2}(\textit{ii}).

(\textit{iv}) The proof repeats the proof of \cite[Theorem 4.3]{KiS}. 

(\textit{iv}) $\Rightarrow$ (\textit{v}) by Gelfand's Theorem.

 (\textit{vi}) The proof repeats the proof of \cite[Theorem 1.3(\textit{v})]{Kin}.
\end{proof}

\begin{remark}
\label{rem_hol}
One can use the operator-valued function $\Theta_p(\zeta,b)$ to \textit{construct} $\Lambda_p(b)$. Then in Theorem \ref{thmLp} one can take $p \in ]p_-,p_+[$, and show that $e^{-t\Lambda_p(b)}$ is holomorphic, see \cite[sect.\,4]{KiS} for details. However, keeping in mind possible extension of this method to more general operators, in this paper we carry out the ``minimal'' argument needed to construct associated Feller semigroup.
\end{remark}

\textbf{3.~}We are in position to complete the proof of Theorem \ref{ThCinfty}(\textit{i}), (\textit{ii}).
Let $m_{d,\alpha}\delta < 4\frac{(d-\alpha)}{(d-\alpha+1)^{2}}$.

The proof follows closely the proof of \cite[Theorem 4.5]{KiS}. Fix $p \in ]d-\alpha + 1, p_+[$. By Theorem \ref{thmLp},
$\Theta_p(\mu,b)(L^p \cap C_\infty) \subset C_\infty$ and $\|\Theta_p(\mu,b)f\|_\infty \leq \mu^{-1}\|f\|_\infty$, $f \in L^p \cap C_\infty$.
Thus, we can define
\begin{equation}
\label{Cinfty_id}
\Theta_{C_{\infty}} (\mu,b)=\big[ \Theta_{p} (\mu,b) \upharpoonright L^{p}\cap C_{\infty}\big]_{C_{\infty} \rightarrow C_{\infty}}^{\rm clos} \in \mathcal{B}(C_{\infty}).
\end{equation}
In several steps:
\begin{center}
1.~ $\mu\Theta_{C_{\infty}} (\mu,b)\overset{s}{\rightarrow} 1$ in $C_{\infty}$ as $\mu \rightarrow \infty$.

\smallskip

2.~$\Theta_{C_\infty}(\mu,b)$ is a pseudo-resolvent on $\{\mu \geq \kappa\lambda\}$.

\smallskip

3.~$(\mu+\Lambda_{C_\infty}(b_n))^{-1} \upharpoonright L^{p}\cap C_{\infty} = \Theta_p(\mu,b_n) \upharpoonright L^{p}\cap C_{\infty}$, $\mu \geq \kappa\lambda$.

\smallskip

4.~$\Theta_p(\mu,b_n)f \overset{s}{\rightarrow} \Theta(\mu,b)f$ in $C_{\infty}$, $\mu \geq \kappa\lambda$, for every $f \in L^{p}\cap C_{\infty}$.
\end{center}

\smallskip

\noindent Steps 1 and 2 yield: By Hille's theorem on pseudo-resolvents \cite[Sect.\,5.2]{HP},  \cite[Ch.\,VIII, sect.\,4]{Yos}, $\Theta_{C_\infty}(\mu,b)$
is the resolvent of a densely defined closed operator $\Lambda_{C_\infty}$ on $C_\infty$. In view of \eqref{Cinfty_id},
$\|(\mu+\Lambda_{C_\infty}(b))^{-1}\|_{\infty \rightarrow \infty} \leq \mu^{-1}$, and since $e^{-t\Lambda_p(b)}$ preserves positivity,  $-\Lambda_{C_\infty}$ generates a Feller generator.
Now, Steps 3 and 4 together with the Trotter Approximation Theorem yield Theorem \ref{ThCinfty}(\textit{i}).
Theorem \ref{ThCinfty}(\textit{ii}) is a direct consequence of Theorem \ref{thmLp}(\textit{ii}).

Let us comment on the proof of 1-4. 
Proof of 1 and 4 follows directly from the definition of $\Theta_p(\mu,b)$ using
$$
\|(\mu  + A)^{-\frac{1}{\alpha}+(-1+\frac{1}{\alpha})\frac{1}{q}}\|_{p \rightarrow \infty} \leq c\mu^{-\frac{1}{\alpha}+\frac{d}{p\alpha}+(-1+\frac{1}{\alpha})\frac{1}{q}}, \quad c< \infty.
$$
Proof of 2 follows from the resolvent identity for $\Theta_p(\mu,b)$. Proof of 3 follows by Theorem \ref{thmLp}.

\section{Proof of Theorem \ref{ThCinfty}(\textit{iii}), (\textit{iv})}

Assertions (\textit{iii}) and (\textit{iv}) of Theorem \ref{ThCinfty} follow immediately from Theorem \ref{thmLp}(\textit{ii}) and Theorem \ref{thmL2}(\textit{i}), respectively, since 
$(\mu+\Lambda_{C_\infty}(b))^{-1} \upharpoonright C_\infty \cap L^p = (\mu+\Lambda_{p}(b))^{-1} \upharpoonright C_\infty \cap L^p $, $(\mu+\Lambda_{C_\infty}(b))^{-1} \upharpoonright C_\infty \cap L^2 = (\mu+\Lambda_{2}(b))^{-1} \upharpoonright C_\infty \cap L^2$.

\section{Proof of Proposition \ref{lem_weights} (weighted estimates)}

\label{weight_sect}

We will use the resolvent representation of Theorem \ref{thmLp} but will consider it in the weighted space   
$$L^p_\eta = L^p(\mathbb R^d, \eta^2 d\mathcal L^d) \quad \text{ with norm }\|\cdot\|_{p,\eta}^p:=\langle |\cdot|^p \eta ^2\rangle, \quad p \in ]p_-,p_+[,
$$
$$
 \eta(x):=(1+|x|^2)^{\nu}, \quad 0<\nu<\frac{\alpha}{2}.$$
Denote $A = (-\Delta)^{\frac{\alpha}{2}}$. It is seen that
$0 \leq A_\eta:=\eta^{-1}A\eta$
is self-adjoint and
$e^{-tA_\eta}=\eta^{-1}e^{-tA}\eta$ in $L^2_\eta$.

\begin{lemma}
\label{markov_lem}
There exists $\omega>0$ such that $\omega+A_{\eta}$ is a symmetric Markov generator on $L^2_\eta$.
\end{lemma}

\begin{proof}
We only need to show that $e^{-t(\omega+A_\eta)}$ is a $L^\infty$ contraction.
By duality, it suffices to prove $$\|\eta e^{-t(\omega+A)}\eta^{-1}f\|_{1} \leq \|f\|_1, \quad f\in L^1.$$
We employ the method introduced in \cite{KSS}. 
Define truncated weights
$
\eta_n:=\theta_n(\eta)$, $n \geq 1,
$
where 
$$
\theta(s):=\left\{
\begin{array}{ll}
s, & 0<s<1, \\
2, & s>2,
\end{array}
\right.
\quad \theta \in C^2(]0,\infty[), \quad \text{ and } \quad \theta_n(s):=n\theta(s/n).
$$
In $L^1$, define  $$Q = Q_n:=\eta_n A_1 \eta_n^{-1},  \;\; D(Q)=\eta_n (1+A)^{-1}L^1, \qquad \text{ and } \quad F_{n}^t:=\eta_n e^{-tA_1} \eta_n^{-1}.$$ Since $\eta_n$, $\eta_n^{-1} \in L^\infty$, these operators are well defined. In particular, $F_n^t$ are bounded $C_0$ semigroups on $L^1$. Denote by $-G = -G_n$ the generator of $F_n^t$, so that $F_n^t=e^{-tG}$.

Let $C_u=\{f \in C(\mathbb R^d) \mid f \text{ is bounded uniformly continuous}\}$. Set
$$
M:=\eta_n(1+A)^{-1}[L^1 \cap C_u]=\eta_n(\mu+A)^{-1}[L^1 \cap C_u], \quad \mu>0.
$$
Then $M \subset D(Q)$ and $M \subset D(G)$. We have $Q \upharpoonright M \subset G$:
$$
Gf=s\mbox{-}L^1\mbox{-}\lim_{t\downarrow 0}t^{-1}(1-e^{-tG})f=\eta_n s\mbox{-}L^1\mbox{-}\lim_{t\downarrow 0}t^{-1}(1-e^{-tA})u=\phi_n A u=Qf.
$$
Thus $Q\upharpoonright M$ is closable and $\tilde{Q}:=(Q\upharpoonright M)^{\rm clos}\subset G$.

\medskip

\textbf{Claim 1.}~The range $R(\mu+\tilde{Q})$, $\mu>0$, is dense in $L^1$. 

\begin{proof}[Proof of Claim 1]Suppose that for some $v \in L^\infty$, $\langle (\mu+\tilde{Q})h,v\rangle=0$ for all $h \in D(\tilde{Q})$. In particular, $\langle (\mu+Q)h,v\rangle=0$ for all $h \in M$, i.e.
$$
\langle (\mu+Q)\eta_n(\mu+A)^{-1}g,v\rangle=0, \quad g \in L^1 \cap C_u,
$$
and so $\langle \eta_n g,v\rangle=0$ for all $g \in L^1 \cap C_u$. The latter clearly implies that $v = 0$, and so $R(\mu+\tilde{Q}$) is dense in $L^1$.
\end{proof}

\textbf{Claim 2.}~There exists $0<\omega\neq \omega(n)$ such that $\omega + \tilde{Q}$ is accretive in $L^1$, i.e.
\begin{equation}
\label{main_est}
\tag{$\bullet$}
\Real\big\langle \big(\omega+ \tilde{Q}\big)f,\frac{f}{|f|}\big\rangle \geq 0, \quad f \in D(\tilde{Q}).
\end{equation}

\begin{proof}[Proof of Claim 2]
For $f=\eta_n u  \in M$, we have
\begin{align*}
\langle Qf,\frac{f}{|f|}\rangle& =\langle \eta_n A_1 u,\frac{f}{|f|}\rangle=\lim_{t\downarrow 0}t^{-1}\langle\eta_n(1-e^{-tA_1})u,\frac{f}{|f|}\rangle \qquad e^{-tA_1}u=e^{-tA_{C_u}}u\\
&=\lim_{t\downarrow 0}t^{-1}\langle\eta_n(1-e^{-tA_{C_u}})u,\frac{f}{|f|}\rangle, \\
\Real\langle Qf,\frac{f}{|f|}\rangle &\geq\lim_{t\downarrow 0}t^{-1}\langle(1-e^{-tA_{C_u}})|u|,\eta_n\rangle\\
&=\lim_{t\downarrow 0}t^{-1}\langle|u|,(1-e^{-tA_{C_u}})\eta_n\rangle\\
&=\langle  |u|,A_{C_u}\eta_n\rangle,
\end{align*}
where, at the last step, we have used that $e^{-tA_{C_u}}$ is a holomorphic semigroup on $C_u$ and so, since, $\eta_n \in D(-\Delta_{C_u})$,
$A\eta_n = A_{C_u}\eta_n$ is well defined and belongs to $C_u$. 

We are going to estimate $J:=\langle  |u|,A\eta_n\rangle $ from below using the representation $(-\Delta)^{\frac{\alpha}{2}}\eta_n=-I_{2-\alpha}\Delta \eta_n$ where $I_{\vartheta}=(-\Delta)^{-\frac{\vartheta}{2}}$ denotes the Riesz potential.
We have
$$
\Delta \eta_n=\theta_n''(\eta)(\nabla \eta)^2 + \theta_n'(\eta)\Delta \eta,
$$
\begin{align*}
\nabla \eta= \nu(1+|x|^2)^{\nu-1}2x, \qquad \Delta \eta(x)=\nu d (1+|x|^2)^{\nu - 2}\left[1+ \left(\frac{2(\nu -1)}{d}+1\right)|x|^2\right].
\end{align*}
Thus, $|\nabla \eta|^2 \leq  C_1(1+|x|^{2})^{2\nu -1}$, $|\Delta \eta (x) | \leq C_2(1+|x|^{2})^{\nu -1}$.
Using that $|\theta_n'| \leq c_1$, and that $\theta_n''$ has support in $\{n<s<2n\}$, $|\theta_n''| \leq c_2/n$, for constants $c_1,c_2$,
we obtain that 
$$
|\Delta \eta_n| \leq C_0  (1+|x|^{2})^{\nu -1}, \quad C_0<\infty.
$$
Now, direct calculations show that, since $0 < \nu < \frac{\alpha}{2}$, 
$$
|I_{2-\alpha}\Delta \eta_n (x)|=\frac{C_0}{\gamma(2-\alpha)}\int_{\mathbb{R}^{d}}\frac{(1+|y|^{2})^{\nu -1}}{|x-y|^{d-2+\alpha}}dy<\infty \quad \text{ for all } x \in \mathbb R^d,
$$
and is continuous in $x$ on any compact set.
Moreover, we have ($|y|>1$)
\begin{align*}
\gamma(2-\alpha)C_0^{-1}\limsup_{r \rightarrow \infty}\sup_{|x|=r}|I_{2-\alpha}\Delta \eta_n (x)|  & \leq \limsup_{r \rightarrow \infty}\sup_{|x|=r}\int_{\mathbb{R}^{d}}\frac{(1+|y|^{2})^{\nu -1}}{|x-y|^{d-2+\alpha}}dy \\
& \text{(put $x=er$ with $e \in \mathbb{R}^{d}$) }\\
& \leq   \limsup_{r \rightarrow \infty} \int_{\mathbb{R}^{d}}\frac{(1+|y|^{2})^{\nu -1}}{|er-y|^{d-2+\alpha}}dy \\
&\leq \limsup_{r \rightarrow \infty} r^{-d+2-\alpha} \int_{\mathbb{R}^{d}}\frac{(1+ r^{2}|z|^{2})^{\nu -1}}{|e-z|^{d-2+\alpha}}r^{d}dz \qquad (z=y/r)\\
& \text{(we are using $r^{2}|z|^2 \leq 1+  r^{2}|z|^{2}$ and $\nu<\frac{\alpha}{2}<1$)}\\
& \leq \limsup_{r \rightarrow \infty} r^{2\nu -\alpha}\int_{\mathbb{R}^{d}}\frac{|z|^{2(\nu -1)}}{|e-z|^{d-2+\alpha}}dz,
\end{align*}
where, clearly, $\int_{\mathbb{R}^{d}}\frac{|z|^{2(\nu -1)}}{|e-z|^{d-2+\alpha}}dz<\infty$. 
It follows that $\limsup_{r \rightarrow \infty}\sup_{|x|=r}|I_{2-\alpha}\Delta \eta_n (x)|=0$ uniformly in $n$.

We conclude that there exists a constant $0<C \neq C(n)$ such that $|I_{2-\alpha}\Delta \eta_n | \leq C $. So, $J \geq -C \|u\|_{1} \geq -C \|\eta_n^{-1}f \|_{1} \geq -C_1\ \|f \|_{1}$.
Putting $\omega = C_1$, we arrive at 
$\Real\langle (\omega + Q)f,\frac{f}{|f|}\rangle  \geq 0$. The latter clearly holds for all $f \in D(\tilde{Q})$, i.e.\,we have proved that $\omega+\tilde{Q}$ is accretive on $L^1$.
\end{proof}

Claims 1 and 2 together with the fact that $\tilde{Q}$ is closed yield: $R(\mu + \tilde{Q})=L^1$. Then by the Lumer-Phillips Theorem $\omega + \tilde{Q}$ is the (minus) generator of a contraction $C_0$ semigroup on $L^1$, $\|e^{-t(\omega + \tilde{Q})}\|_{1 \rightarrow 1} \leq 1$. Since $\tilde{Q} \subset G$, this semigroup must coincide with $F^t_n=\eta_n e^{-tA_1} \eta_n^{-1}$. It follows that 
\begin{equation*}
\|\eta_n e^{-tA} \eta_n^{-1}f\|_{1} \leq e^{\omega t}\|f\|_1, \quad f \in L^1,
\end{equation*}
and so, using e.g.\,Fatou's Lemma, we obtain 
$$\|\eta e^{-tA} \eta^{-1}f\|_{1} \leq e^{\omega t}\|f\|_1.$$
The proof of Lemma \ref{markov_lem} is completed.
\end{proof}

\begin{lemma}
\label{claim2}
Let $b \in  \mathbf{F}_\delta^{\scriptsize \frac{\alpha-1}{2}}$, $\delta>0$. Then for every $\mu > \kappa(\omega \vee \lambda)$ and all $p \in ]1,\infty[$,
$$
G_{p,\eta}(r):=\eta^{-1} b^{\frac{1}{p}}\cdot \nabla (\mu+A)^{(-1+\frac{1}{\alpha})\frac{1}{r}}\eta\in \mathcal B(L^p_\eta), \quad r<p,$$ 
the operators
$$T_{p,\eta}:=\eta^{-1}b^{\frac{1}{p}} \cdot \nabla  (\mu+A)^{-1}|b|^{\frac{1}{p'}}\eta \quad \text{ on } \mathcal E=\bigcup_{\varepsilon>0}e^{-\varepsilon|b|}L^p_\eta,$$
$$ Q_{p,\eta}(q):=(\mu+A_\eta)^{(-1+\frac{1}{\alpha})\frac{1}{q'}}|b|^{\frac{1}{p'}} \quad \text{ on } \mathcal E,
\quad q>p,
$$
admit extension by continuity to $\mathcal B(L^p_\eta)$, and
$$
\|T_{p,\eta}\|_{p,\eta \rightarrow p,\eta}  \leq m_{d,\alpha} c_p\delta,
$$
$$
\|G_{p,\eta}(r)\|_{p,\eta \rightarrow p,\eta} \leq M_{1,r}, \quad \|Q_{p,\eta}(q)\|_{p,\eta \rightarrow p,\eta} \leq M_{2,q}
$$
for $M_{1,r} \neq M_{1,r}(\mu)$, $M_{2,q} \neq M_{2,q}(\mu)$.
\end{lemma}
\begin{proof}
Let us note that
$b \in \mathbf{F}_\delta^{\scriptsize \frac{\alpha-1}{2}}$ is equivalent to
\begin{equation*}
\||b|^{\frac{1}{2}}(\lambda+A_\eta)^{-\frac{1}{2}+\frac{1}{2\alpha}}\|_{2,\eta \rightarrow 2,\eta} \leq \sqrt{\delta}, \quad  \lambda=\lambda_\delta.
\end{equation*}
Thus, in view of Lemma \ref{markov_lem}, we can apply Lemma \ref{abc_lem} (Appendix \ref{lp_sect}) to obtain
\begin{align}
\||b|^{\frac{1}{p}}(\mu+A_\eta)^{-1+\frac{1}{\alpha}}\|_{p,\eta \rightarrow p,\eta} & \leq (c_{p}\delta)^{1/p}\mu^{(-1+\frac{1}{\alpha})/p'}, \tag{$\mathbf a'$} \label{a_eta} \\
\||b|^{\frac{1}{p}}(\mu+A_\eta)^{-1+\frac{1}{\alpha}}|b|^{\frac{1}{p'}}\|_{p,\eta \rightarrow p,\eta} & \leq c_p\delta, \tag{$\mathbf b'$} \label{b_eta} \\
\|(\mu +A_\eta )^{-1+\frac{1}{\alpha}}|b|^{\frac{1}{p'}}f\|_{p,\eta \rightarrow p,\eta}  &\leq (c_{p}\delta)^{1/p'}\mu^{(-1+\frac{1}{\alpha})/p} \tag{$\mathbf c'$} \label{c_eta}. 
\end{align}
Now, using pointwise estimates \eqref{A} and \eqref{estimation2} as in the proof of Theorem \ref{thmLp}(\textit{i}), we obtain the assertion of Lemma \ref{claim2}. For example, let us show that 
$\|G_{p,\eta}(r)\|_{p,\eta \rightarrow p,\eta} \leq M_{1,r}$.
We will need the formula
\begin{equation}
\label{eq_eta}
\tag{$\ast\ast'$}
(\mu+A_\eta)^{-\tau}=\frac{\sin \pi \tau}{\pi}\int_{0}^{\infty}t^{-\tau}(t+\mu+A_\eta)^{-1}dt, \quad 0 < \tau < 1.
\end{equation}
Let  $f \in C_c^\infty$ and  $\mu > \kappa(\lambda \vee \omega)$. We have for all $r < p$ (without loss of generality, $\kappa \geq 1$)
\begin{align*}
\|G_{p,\eta}(r)f\|_{p,\eta}&=\|\eta^{-1} b^{\frac{1}{p}}\cdot \nabla (\mu + A)^{-1/\alpha +(-1+\frac{1}{\alpha})/r} \eta f\|_{p,\eta}\\
&(\text{we are using \eqref{estimation2}})\\
& \leq  c_{d,\gamma}\|\eta^{-1}|b|^{\frac{1}{p}}(\mu + A)^{(-1+\frac{1}{\alpha})/r}\eta f\|_{p,\eta}= c_{d,\gamma}\||b|^{\frac{1}{p}}(\mu + A_\eta)^{(-1+\frac{1}{\alpha})/r} f\|_{p,\eta}\\
& (\text{we are using  \eqref{eq_eta}})\\
&\leq   c_{d,\gamma}k_{\alpha, r}\int_{0}^{\infty}t^{(-1+\frac{1}{\alpha})/r}\||b|^{\frac{1}{p}}(t+\mu + A_\eta)^{-1}|f|\|_{p,\eta}dt\\
& \leq  c_{d,\gamma}k_{\alpha, r}\int_{0}^{\infty}t^{(-1+\frac{1}{\alpha})/r}\||b|^{\frac{1}{p}}(t+\mu + A_\eta)^{-1+\frac{1}{\alpha}}\|_{p,\eta \rightarrow p,\eta} \|(t+\mu +A_\eta)^{-\frac{1}{\alpha}}|f|\|_{p,\eta} dt \\
&(\text{we are using \eqref{a_eta}})\\
& \leq  c_{d,\gamma}k_{\alpha, r}(c_{p}\delta)^{1/p}\left(\int_{0}^{\infty}t^{(-1+\frac{1}{\alpha})/r}(t+\mu-\omega)^{-\frac{1}{\alpha}+(-1+\frac{1}{\alpha})/p'}dt\right) \|f\|_{p,\eta}\\
&=  M_{1,r} \|f\|_{p,\eta},
\end{align*}
where $M_{1,r} < \infty$ because $r < p$. 
\end{proof}

By  Lemma \ref{claim2}, if $p \in ]p_-,p_+[$, then $\|T_{p,\eta}\|_{p,\eta \rightarrow p,\eta} \leq m_{d,\alpha} c_p\delta <1$. Using the resolvent representation of Theorem \ref{thmLp}(\textit{i}), we obtain for all $\mu \geq \kappa\lambda$, $p \in [2,p_+[$
\begin{align}
\label{repr_eta}
\eta^{-1}(\mu+\Lambda_p(b))^{-1}\eta&=(\mu+A_\eta)^{-1} \tag{$\star\star$} \\
&- (\mu+A_\eta)^{-\frac{1}{\alpha}+(-1+\frac{1}{\alpha})\frac{1}{q}} Q_{p,\eta}(q)(1+T_{p,\eta})^{-1}G_{p,\eta}(r)(\mu+A_\eta)^{(-1+\frac{1}{\alpha})\frac{1}{r'}} \in \mathcal B(L^p_\eta). \notag
\end{align}

We will need 

\begin{lemma}
\label{lem_a}
Assume that $\delta<m_{d,\alpha}^{-1}4\bigl[\frac{d-\alpha}{(d-\alpha+1)^2} \wedge \frac{\alpha(d+\alpha)}{(d+2\alpha)^2}\bigr]$. Then there exist a $\nu<\frac{\alpha}{2}$ close to $\frac{\alpha}{2}$, a $p \in \big](d-\alpha+1) \vee (\frac{d}{2\nu}+2),p_+\big[$ and
a $q>p$ close to $p$ such that
$$\|(\mu+A_\eta)^{-\frac{1}{\alpha}+(-1+\frac{1}{\alpha})\frac{1}{q}}h\|_\infty \leq C\|h\|_{p,\eta}, \quad h \in C_c,$$
for  a constant $C=C(\alpha, \nu, q, p)$.
\end{lemma}
\begin{proof}
Set $\tau:=\frac{1}{\alpha}+(1-\frac{1}{\alpha})\frac{1}{q}~(<1)$. 
Below we use well known estimate
$$
(1+A)^{-\tau}(x,y) \leq c \bigl[|x-y|^{-d+\alpha \tau} \wedge |x-y|^{-d-\alpha}\bigr], \quad x,y \in \mathbb R^d, \quad x \neq y.
$$
We have:
\begin{align*}
(\mu+A_\eta)^{-\tau}h(x) 
&=\eta^{-1}(x) \langle (\mu+A)^{-\tau}(x-y)\eta(y)|h(y)|\rangle_y \\
&\leq c \eta^{-1}(x) \langle |x-y|^{-d-\alpha} \wedge |x-y|^{-d+\alpha\tau} \eta(y)|h(y)|\rangle_y \\
&=c \eta^{-1}(x) \langle |x-y|^{-d+\alpha\tau} \mathbf{1}_{B(x,1)}(y) \eta(y)|h(y)|\rangle_y \\
&+c \eta^{-1}(x) \langle |x-y|^{-d-\alpha} \mathbf{1}_{B^c(x,1)}(y) \eta(y)|h(y)|\rangle_y =:S_1(x)+S_2(x).
\end{align*}
Note that $\|S_1\|_\infty \leq C_1\|h\|_{p,\eta}$. Indeed,
\begin{align*}
S_1(x) \leq c\bigg(\eta^{-1}(x)\sup_{y \in B(x,1)}\eta(y)\bigg)  \langle |x-y|^{-d+\alpha\tau}\mathbf{1}_{B(x,1)}(y) |h(y)|\rangle_y,
\end{align*}
where $\eta^{-1}(x)\sup_{y \in B(x,1)}\eta(y)$ is in $L^\infty$ and, since $p>d-\alpha+1$, we have for every $x \in \mathbb R^d$ by H\"{o}lder's inequality $ \langle |x-y|^{-d+\alpha\tau}\mathbf{1}_{B(x,1)}(y) |h(y)|\rangle_y= \langle |y|^{-d+\alpha\tau}\mathbf{1}_{B(0,1)}(y) |h(x+y)|\rangle_y \leq C_S\|h\|_p$. Since $\|h\|_p \leq \|h\|_{p,\eta}$, we obtain the required.

Next, 
\begin{align*}
S_2(x) & \leq c\eta^{-1}(x) \||x-\cdot|^{-d-\alpha}\mathbf{1}_{B^c(x,1)}(\cdot)\eta^{1-\frac{2}{p}}(\cdot)\|_{p'} \|\eta^{\frac{2}{p}}h\|_{p}, \qquad \|\eta^{\frac{2}{p}}h\|_{p} =\|h\|_{p,\eta} \\
& =:c K(x)\|h\|_{p,\eta} .
\end{align*}
Thus, it remains to show that
$
K(x) =\eta^{-1}(x)\||\cdot|^{-d-\alpha}\mathbf{1}_{B^c(0,1)}(\cdot)\eta^{1-\frac{2}{p}}(x+\cdot) \|_{p'}
$
 is in $L^\infty$.

For $|x|\leq 1$ this is immediate since $d+\alpha - 2\nu\left(1-\frac{2}{p}\right)>\frac{d}{p'}$.  

For $|x|>1$, we estimate
\begin{align*}
K(x)  \leq c_0|x|^{-2\nu}\||\cdot|^{-d-\alpha}\mathbf{1}_{B^c(0,1)}(\cdot)|x+\cdot|^{2\nu(1-\frac{2}{p})}\|_{p'}, \quad c_0>0. 
\end{align*}
Thus, writing $x=er$, $|e|=1$, $r>1$, we have
\begin{align*}
K(er) & \leq C r^{-2\nu} \||\cdot|^{-d-\alpha}\mathbf{1}_{B^c(0,1)}(\cdot)|x+\cdot|^{2\nu(1-\frac{2}{p})}\|_{p'} \\
& = C r^{-2\nu} r^{-d-\alpha + 2\nu(1-\frac{2}{p}) + \frac{d}{p'}}\| |\cdot|^{-d-\alpha}\mathbf{1}_{B^c(0,r^{-1})}(\cdot)|e+\cdot|^{2\nu(1-\frac{2}{p})}\|_{p'},
\end{align*}
where the second multiple
\begin{align*}
\| |\cdot|^{-d-\alpha}\mathbf{1}_{B^c(0,r^{-1})}|e+\cdot|^{2\nu(1-\frac{2}{p})}\|_{p'}  
&\leq c_1 + 
\| |\cdot|^{-d-\alpha}\mathbf{1}_{B(0,1)-B(0,r^{-1})}(\cdot)|e+\cdot|^{2\nu(1-\frac{2}{p})}\|_{p'} \\
& \leq c_1 + c_2 \| |\cdot|^{-d-\alpha}\mathbf{1}_{B(0,1)-B(0,r^{-1})}(\cdot)\|_{p'} \leq c_3 r^{d+\alpha-\frac{d}{p'}},
\end{align*}
and so
$
K(er)\leq C c_3 r^{-\frac{4\nu}{p}}.
$
Thus, $K(er)$ is bounded in $r>1$, and hence $\|S_2\|_\infty \leq C_2\|h\|_{p,\eta}$. 

The proof of Lemma \ref{lem_a} is completed.
\end{proof}

We are in position to complete the proof of Proposition \ref{weight_lem}. 

Lemma \ref{lem_a} applied in the resolvent representation \eqref{repr_eta} yields
$$
\|\eta^{-1} (\mu+\Lambda_p(b))^{-1}\eta f\|_\infty \leq K_1\|f\|_{p,\eta}, \quad f \in C_c,
$$
and thus yields \eqref{j_1_w}.
Now, taking into account that 
$
G_{p,\eta}(r)(\mu+A_\eta)^{(-1+\frac{1}{\alpha})\frac{1}{r'}}|b_m| = T_{p,\eta}|b_m|^{\frac{1}{p}},
$
we obtain \eqref{j_2_w}.
\eqref{j_3_w} follows immediately from \eqref{repr_eta} and Lemma \ref{claim2}.

\section{Corollary of Proposition \ref{lem_weights}}

First, we prove the following elementary consequence of $b \in \mathbf{F}_\delta^{\frac{\alpha-1}{2}}$.

\begin{lemma}
\label{p_eta_lem}
$\|\eta^{-1}|b|^{\frac{1}{p}}\|_{p,\eta}<\infty$, $p>\frac{d}{2\nu}+2.$
\end{lemma}
\begin{proof}
We have $\|\eta^{-1}|b|^{\frac{1}{p}}\|_{p,\eta}^{p}=\|\eta^{-\frac{p}{2}+1}|b|^{\frac{1}{2}}\|^2_{2}$ and
\begin{align*}
\|\eta^{-\frac{p}{2}+1}|b|^{\frac{1}{2}}\|_{2} 
& = \||b|^{\frac{1}{2}}(\lambda+A)^{-\frac{\alpha-1}{2\alpha}}(\lambda+A)^{\frac{\alpha-1}{2\alpha}} \eta^{-\frac{p}{2}+1} \|_{2} \\
& \text{(we are using $b \in \mathbf{F}_\delta^{\scriptsize \frac{\alpha-1}{2}}$)} \\
&  \leq \sqrt{\delta} \|(\lambda+A)^{\frac{\alpha-1}{2\alpha}}\eta^{-\frac{p}{2}+1}\|_{2},
\end{align*}
where 
\begin{align*}
\|(\lambda+A)^{\frac{\alpha-1}{2\alpha}}\eta^{-\frac{p}{2}+1}\|_{2} & \leq \|(\lambda+A)^{\frac{\alpha-1}{2\alpha}}(1-\Delta)^{-\frac{\alpha-1}{4}}\|_{2 \rightarrow 2} \|(1-\Delta)^{\frac{\alpha-1}{4}} \eta^{-\frac{p}{2}+1}\|_2 \\
& (\text{by the Spectral Theorem, $\|(\lambda+A)^{\frac{\alpha-1}{2\alpha}}(1-\Delta)^{-\frac{\alpha-1}{4}}\|_{2 \rightarrow 2} \leq C<\infty$}) \\
& \leq C \|(1-\Delta)^{\frac{\alpha-1}{4}} \eta^{-\frac{p}{2}+1}\|_2 \\
& \leq C\|(1-\Delta)^{-1+\frac{\alpha-1}{4}} (1-\Delta)\eta^{-\frac{p}{2}+1}\|_2<\infty
\end{align*}
since $(1-\Delta)\eta^{-\frac{p}{2}+1} \in L^2$  by $p>\frac{d}{2\nu}+2$.
\end{proof}

Put by definition
\begin{equation*}
(e^{-t\Lambda_{C_\infty}(b)}b\cdot g)(x):=\big\langle e^{-t\Lambda_{C_\infty}(b)}(x,\cdot)b(\cdot)\cdot g(\cdot) \big\rangle, \quad g \in L^\infty(\mathbb R^d,\mathbb R^d).
\end{equation*}

The next result is a consequence of Proposition \ref{lem_weights}. Set $\mathbb R_+:=[0,\infty[$.

\begin{corollary}
\label{aux_prop}
In the assumptions of Proposition \ref{weight_lem}, there exist $\nu<\frac{\alpha}{2}$ close to $\frac{\alpha}{2}$ and $p>(d-\alpha +1) \vee (\frac{d}{2\nu}+2)$ 
such that, for every $w \in L^\infty(\mathbb R_+ \times \mathbb R^d,\mathbb R^d)$, we have (write $w=w(s,\cdot)$)

{\rm(\textit{i})} For every $h \in L^\infty(\mathbb R^d)$
\begin{equation*}
\big\|\eta^{-1}\int_0^t e^{-s\Lambda_{C_\infty}(b)}b \cdot h w ds \big\|_\infty \leq e^{\kappa\lambda t}K_2 \big\|\eta^{-1}|b|^{\frac{1}{p}} h \big\|_{p,\eta}\|w\|_{L^\infty(\mathbb R_+  \times \mathbb R^d,\mathbb R^d)}<\infty.
\end{equation*}

\smallskip

{\rm(\textit{ii})} 
$
\int_0^t e^{-s\Lambda_{C_\infty}(b_n)}b_n \cdot w ds \rightarrow \int_0^t e^{-s\Lambda_{C_\infty}(b)}b\cdot w ds$ locally uniformly on $\mathbb R_+ \times \mathbb R^d$.

\smallskip

{\rm(\textit{iii})} $\int_0^t e^{-s\Lambda_{C_\infty}(b)}b\cdot w ds$ is continuous on $\mathbb R_+ \times \mathbb R^d$.

\smallskip

{\rm(\textit{iv})} The Duhamel formula:
$$
e^{-t\Lambda_{C_\infty}(b)}(x,y)=e^{-t(-\Delta)^{\frac{\alpha}{2}}}(x,y) + \int_0^t \bigg\langle e^{-(t-s)\Lambda_{C_\infty}(b)}(x,\cdot)\,b(\cdot) \cdot \nabla_{\cdot} e^{-s(-\Delta)^{\frac{\alpha}{2}}}(\cdot,y)\bigg\rangle ds.
$$

{\rm(\textit{v})} 
$\big\|\eta^{-1}e^{-t\Lambda_{C_\infty}(b)} h \big\|_\infty \leq e^{\kappa\lambda t}K_2 \|\eta^{-1}(1+|b|^{\frac{1}{p}}) h\|_{p,\eta}<\infty$, $h \in L^\infty(\mathbb R^d)$.

\end{corollary}

\begin{proof}
(\textit{i}) Using Proposition \ref{weight_lem}($E_2$), we estimate
\begin{align*}
&\big\|\eta^{-1}\int_0^t e^{-s\Lambda_{C_\infty}(b_n)}b_n \cdot h w ds \big\|_{L^\infty(\mathbb R^d)} \\
&\leq e^{\mu t}\big\|\eta^{-1}\big(\mu+\Lambda_{C_\infty}(b_n)\big)^{-1}|b_n|h\big\|_{L^\infty(\mathbb R^d)} \|w\|_{L^\infty(\mathbb R_+  \times \mathbb R^d,\mathbb R^d)}  \\
&\leq e^{\mu t}K_2 \|\eta^{-1}|b|^{\frac{1}{p}}h \|_{p,\eta} \|w\|_{L^\infty(\mathbb R_+ \times \mathbb R^d,\mathbb R^d)}<\infty.
\end{align*} 
Now Fatou's Lemma yields (\textit{i}).

\smallskip

(\textit{ii}) First, let us prove that, for every $w \in L_{\rm com}^\infty(\mathbb R_+  \times \mathbb R^d,\mathbb R^d)$ (the vector fields with entries in $L^\infty$ having compact support),
\begin{equation}
\label{conv_f}
\tag{$\ast$}
\left\|\int_0^t e^{-s\Lambda_{C_\infty}(b)}b \cdot w ds - \int_0^t e^{-s\Lambda_{C_\infty}(b_n)}b_n\cdot w ds \right\|_{L^\infty(\mathbb R^d)} \rightarrow 0
\end{equation}
locally uniformly in $t \in \mathbb R_+$.

Step 1. 
\begin{equation}
\label{ast1}
\tag{$\ast\ast$}
J:=\left\|\int_0^t  e^{-s\Lambda_{C_\infty}(b)}(b-b_n) \cdot w ds \right\|_{L^\infty(\mathbb R^d)} \rightarrow 0
\end{equation}
locally uniformly in $t \in \mathbb R_+$. 

Indeed, fix some $g \in L^\infty_{\rm com}(\mathbb R^d)$ such that $g \geq |w|$ a.e.\,on $\mathbb R_+  \times \mathbb R^d$. Then (recall $\|\cdot\|_p=\|\cdot\|_{L^p(\mathbb R^d)}$)
\begin{align*}
J 
&\leq e^{\mu t}\|(\mu+\Lambda_{p}(b))^{-1}|b-b_n|g\|_{L^\infty(\mathbb R^d)} \\
& \leq
 e^{\mu t} C_1 \||b-b_n|^{\frac{1}{p}}g\|_p \rightarrow 0 \text{ as } n \rightarrow \infty \quad \text{(by $|b| \in L^1_{\loc}(\mathbb R^d)$)}.
\end{align*}
(Here $(\mu+\Lambda_{p}(b))^{-1}|b-b_n|(x):=\langle (\mu+\Lambda_{p}(b))^{-1}(x,y)|b(y)-b_n(y)| \rangle_y=\lim_k\langle (\mu+\Lambda_{p}(b))^{-1}(x,y)|b_k(y)-b_n(y)| \rangle_y$.)

Step 2. 
\begin{equation}
\label{ast2}
\tag{$\ast\ast\ast$}
\left\|\int_0^t e^{-s\Lambda_{C_\infty}(b)}b_n \cdot w ds - \int_0^t e^{-s\Lambda_{C_\infty}(b_n)}b_n \cdot w ds\right\|_{L^\infty(\mathbb R^d)} \rightarrow 0
\end{equation}
locally uniformly in $t \in \mathbb R_+$.

Indeed, write 
\begin{align*}
&\int_0^t e^{-s\Lambda_{C_\infty}(b)}b_n \cdot w ds - \int_0^t e^{-s\Lambda_{C_\infty}(b_n)}b_n \cdot w ds \\
& =\int_{0}^{t}(e^{-s\Lambda_{C_{\infty}}(b)}-e^{-s\Lambda_{C_{\infty}}(b_{n})})(b_{n}-b_{m}) \cdot w ds+ \int_{0}^{t}(e^{-s\Lambda_{C_{\infty}}(b)}-e^{-s\Lambda_{C_{\infty}}(b_{n})})b_{m} \cdot w ds \\ &=:R_{1}+R_{2}.
\end{align*}
where $m$ is to be chosen.
Arguing as above we obtain, for every $x \in \mathbb R^d$,
\begin{align*}
R_{1}(x) & \leq e^{\mu t}(\mu + \Lambda_{C_{\infty}}(b))^{-1}|b_{n}-b_{m}| g(x) +e^{\mu t}(\mu + \Lambda_{C_{\infty}}(b_{n}))^{-1}|b_{n}-b_{m}| g(x)  \\
&\leq C e^{\mu t}\|(b_n-b_m)^{\frac{1}{p}}g\|_p \rightarrow 0 \text{ as }\quad n,m \rightarrow \infty.
\end{align*}
To estimate $R_2(x)$, fix $m$ sufficiently large. Let $s \in ]0,t]$. For every $\varepsilon>0$, using Lusin's Theorem (recall that $b_m$ has compact support), we can write $b_m \cdot w(s)=b_m \cdot w'+b_m \cdot w''$, where $w' \in C_c(\mathbb R^d)$, $w'' \in L^\infty_{\rm com}(\mathbb R^d)$ with $\|w'\|_{L^\infty(\mathbb R^d)} \leq \|w(s)\|_{L^\infty(\mathbb R^d)}$, $\|w''\|_{L^r(\mathbb R^d)}<\varepsilon$, $r \geq 2$. By Theorem \ref{thmLp}(\textit{i}), for each $s>0$,
$$
\|(e^{-s\Lambda_{C_{\infty}}(b)}-e^{-s\Lambda_{C_{\infty}}(b_{n})})b_{m} \cdot w'(s)\|_{L^\infty(\mathbb R^d)} \rightarrow 0 \quad \text{ as } n \rightarrow \infty,
$$
and by Theorem \ref{thmLp}(\textit{iv}), for all $n \geq 1$,
$$ 
\|e^{-s\Lambda_{C_{\infty}}(b)}b_m \cdot w''\|_{L^\infty(\mathbb R^d)}, \|e^{-s\Lambda_{C_{\infty}}(b_n)}b_m \cdot w''\|_{L^\infty(\mathbb R^d)} \leq \varepsilon c_r e^{\omega s}s^{-\frac{d}{\alpha r}}\|b_m\|_{L^\infty(\mathbb R^d)}.
$$
Therefore, for every $s \in ]0,t]$, 
$$
\|(e^{-s\Lambda_{C_{\infty}}(b)}-e^{-s\Lambda_{C_{\infty}}(b_{n})})b_{m} \cdot w(s)\|_{L^\infty(\mathbb R^d)} \rightarrow 0 \quad \text{ as } n \rightarrow \infty.
$$
Finally, appealing to the Dominated Convergence Theorem, 
we obtain that $\|R_2\|_{L^\infty(\mathbb R^d)} \rightarrow 0$ locally uniformly in $t \in \mathbb R_+$.
We have proved \eqref{ast2}.

Now, \eqref{ast1}, \eqref{ast2} yield \eqref{conv_f}.

\smallskip

Armed with \eqref{conv_f}, we now complete the proof (\textit{ii}). Let $w \in L^\infty(\mathbb R_+  \times \mathbb R^d,\mathbb R^d)$. Write  $$w=w_1+w_2, \quad w_1:=\mathbf{1}_{B(0,R)}w, \quad w_2:=\mathbf{1}_{B^c(0,R)} w, \quad R>0.$$ By (\textit{i}),
\begin{align*}
\sup_n\big\|\eta^{-1}\int_0^t e^{-s\Lambda_{C_\infty}(b_n)}b_n \cdot w_2 ds \big\|_{L^\infty(\mathbb R^d)}, \;& \big\|\eta^{-1}\int_0^t e^{-s\Lambda_{C_\infty}(b)}b \cdot w_2 ds \big\|_{L^\infty(\mathbb R^d)} \\
&\leq e^{\mu t}K_2 \|\eta^{-1}|b|^{\frac{1}{p}}\mathbf{1}_{B^c(0,R)}\|_{p,\eta}\|w\|_{L^\infty(\mathbb R_+ \times \mathbb R^d,\mathbb R^d)},
\end{align*}
where, in view of Lemma \ref{p_eta_lem}, the RHS can be made arbitrarily small by selecting $R$ sufficiently large. In turn, by \eqref{conv_f}
$$
\left\|\int_0^t e^{-s\Lambda_{C_\infty}(b)}b \cdot w_1 ds - \int_0^t e^{-s\Lambda_{C_\infty}(b_n)}b_n \cdot w_1 ds\right\|_{L^\infty(\mathbb R^d)} \rightarrow 0
$$
locally uniformly in $t \in \mathbb R_+$. The yields (\textit{ii}).

\smallskip

(\textit{iii}) It suffices to prove that $\int_0^t e^{-s\Lambda_{C_\infty}(b_n)}b_n\cdot w ds$ is continuous on $\mathbb R_+  \times \mathbb R^d$ and then apply (\textit{ii}). 

To prove the former, we note that for every $s>0$, $b_n(\cdot) \cdot w(s,\cdot)$ is bounded and has compact support. Thus, by Theorem \ref{thmLp}, $e^{-s\Lambda_{C_\infty}(b_n)}b_n\cdot  w(s) \in C_\infty$, so $\int_0^t e^{-s\Lambda_{C_\infty}(b_n)}b_n \cdot w ds$ is continuous on $\mathbb R_+  \times \mathbb R^d$.

\smallskip

(\textit{iv}) It suffices to prove the Duhamel formula on test functions $\varphi \in C_c^\infty$. By the Duhamel formula for $\Lambda_{C_\infty}(b_n)$
we have for every $f \in C_c^\infty$
$$
\langle e^{-t\Lambda_{C_\infty}(b_n)}f, \varphi \rangle =\langle e^{-t(-\Delta)^{\frac{\alpha}{2}}}f,\varphi \rangle  + \left\langle \int_0^t  e^{-(t-s)\Lambda_{C_\infty}(b_n)}\,b_n \cdot \nabla e^{-s(-\Delta)^{\frac{\alpha}{2}}}f  ds, \varphi \right\rangle.
$$
It remains to apply Theorem \ref{ThCinfty}(\textit{i}) and assertion (\textit{ii}) proved above with $w=\nabla e^{-s(-\Delta)^{\frac{\alpha}{2}}}f$.

\smallskip

(\textit{v}) is obtained by applying consecutively (\textit{iv}), (\textit{i}) and Lemma \ref{p_eta_lem}.
\end{proof}

\section{Proof of Theorem \ref{ThCinfty}(\textit{v}), (\textit{vi})}

Let $\mathbb P^n_x$  be the probability measures associated with $e^{-t\Lambda_{C_\infty}(b_n)}$, $n \geq 1$.
Set $\mathbb E_x:=\mathbb E_{\mathbb P_x}$, and $\mathbb E^n_x:=\mathbb E_{\mathbb P_x^n}$.

First, we note that for every $x \in \mathbb{R}^{d}$ and $t > 0$, $b_{n}(X_t) \rightarrow b(X_t)$, $\mathbb{P}_{x}$ a.s. as $n \uparrow \infty$. 
Indeed, by (\ref{Cinfty_id})  and the Dominated Convergence Theorem, for any $\mathcal{L}^{d}$ measure zero set $G \subset \mathbb{R}^{d}$ and every $t > 0$ , $\mathbb{P}_{x}[X_t \in  G] = 0$. Since $b_{n} \rightarrow b$ a.e.\,we have the required.

Fix $x \in \mathbb R^d$.

(\textit{v}) To show that the semigroup $e^{-t\Lambda_{C_\infty}(b)}$ is conservative, it suffices to show that 
\begin{equation}
\label{conv_7}
\mathbb E_x[\xi_k(X_t)] \rightarrow 1 \quad \text{ as }k \uparrow \infty, 
\end{equation}
where
\begin{equation}
\label{xi_k}
\xi_k(y):=\left\{
\begin{array}{ll}
\upsilon (|y|+1-k) & |y| \geq k, \\
1 & |y|<k,
\end{array}
\right.
\end{equation}
for a fixed $\upsilon \in C^\infty([0,\infty[)$, $\upsilon (s)=1$ if $0 \leq s \leq 1$, $\upsilon (s)=0$ if $s \geq 2$. 

By Theorem \ref{ThCinfty}(\textit{i}),
$\mathbb E_x[\xi_k(X_t)]=\lim_{n} \mathbb E^n_x[\xi_k(X_t)]$ uniformly on every compact interval of $t \geq 0$, so convergence \eqref{conv_7} would follow from
\begin{equation}
\label{conv_n9}
\mathbb E^n_x[\xi_k(X_t)] \rightarrow 1 \quad \text{ as $k \uparrow \infty$ \textit{uniformly in} $n$}.
\end{equation}
In turn,
since $\mathbb E^n_x[\mathbf{1}_{\mathbb R^d}(X_t)]=1$ for all $n=1,2,\dots$, \eqref{conv_n9} is equivalent to
$\mathbb E^n_x[(\mathbf{1}_{\mathbb R^d}-\xi_k)(X_t)]  \rightarrow 0$ as $k \uparrow \infty$ uniformly in $n$. 
We have by the Dominated Convergence Theorem
\begin{align*}
 \mathbb E^n_x[(\mathbf{1}_{\mathbb R^d}-\xi_k)(X_t)] 
& = \lim_{r \uparrow \infty} \mathbb E^n_x[\xi_r(1-\xi_k)(X_t)] \\
& \text{(we are using Corollary \ref{aux_prop}(\textit{v}))} \\
& \leq \eta(x) K_2 e^{\kappa\lambda t} \lim_{r \uparrow \infty} \|\eta^{-1}\xi_r (1-\xi_k)\|_{p,\eta} \\
&\leq \eta(x) K_2 e^{\kappa\lambda t} \|\eta^{-1} (1-\xi_k) \|_{p,\eta} \rightarrow 0 \quad \text{as $k \rightarrow \infty$},
\end{align*}
where at the last step we have used $\|\eta^{-1}\|_{p,\eta}<\infty$ since $p>\frac{d}{2\nu}+2$.

\medskip

(\textit{vi}) 
By Fatou's Lemma,
\begin{align*}
\mathbb{E}_{x}\int_{0}^{t}|b(X_s)|ds & \leq \liminf_{n}\mathbb{E}_{x} \int_{0}^{t} |b_{n}(X_s)|ds =
\liminf_{n} \int_{0}^{t}e^{-s\Lambda_{C_{\infty}}(b)}|b_{n}|(x)ds \\
& (\text{we argue as in the proof of Corollary \ref{aux_prop}(\textit{ii})}) \\
& \leq  K_{2} e^{\kappa\lambda t} \eta(x) \|\eta^{-1}|b|^{\frac{1}{p}}\|_{p,\eta}<\infty \quad (\text{Lemma \ref{p_eta_lem}}).
\end{align*}

\section{Proof of Theorem \ref{ThCinfty}(\textit{vii})}

\label{sde_sect}

We follow the approach of \cite{PP,P} (see also \cite{CW}) but in appropriate weighted space. Set
$$
Z_t:=X_t-X_0-\int_0^t b(X_s)ds, \quad t \geq 0.
$$
Our goal is to prove that under $\mathbb P_x$ the process $Z_t$ is a symmetric $\alpha$-stable process starting at $0$. We use notation introduced in the beginning of the previous section. 
For brevity, write $e^{-t\Lambda(b)} = e^{-t\Lambda_{C_\infty}(b)}$.

\smallskip

1. \textit{Define
\begin{equation}
\label{w}
w(t,x,\varkappa)=\mathbb E_x \biggl[e^{i \varkappa \cdot \bigl (X_t-\int_0^t b(X_s)ds\bigr) } \biggr], \quad t \geq 0, \quad \varkappa \in \mathbb R^d.
\end{equation} 
Then $w$ is a bounded solution to integral equation}
\begin{equation}
\label{eq_w}
w(t,x,\varkappa)=\int_{\mathbb R^d} e^{i \varkappa \cdot y}e^{-t\Lambda(b)}(x,y) dy - i\int_0^t \int_{\mathbb R^d} e^{-(t-s)\Lambda(b)}(x,z)(\varkappa \cdot b(z))w(s,z,\varkappa)dzds.
\end{equation}
Indeed, in view of
\begin{equation*}
e^{-i\cdot\varkappa\int_0^t b(X_\tau)d\tau}=1-i\int_{0}^{t}(\varkappa \cdot b(X_s))e^{-i\cdot\varkappa\int_s^t b(X_\tau)d\tau},
\end{equation*}
one has
\begin{align*}
w(t,x,\varkappa)&
=\mathbb E_x \biggl[e^{i \varkappa \cdot X_t}\biggr]-i\int_{0}^{t}\mathbb E_x \biggl[e^{i \varkappa \cdot X_t}(\varkappa \cdot b(X_s))e^{-i\cdot\varkappa\int_s^t b(X_\tau)d\tau}\biggr]ds\\
&=\mathbb E_x \biggl[e^{i \varkappa \cdot X_t}\biggr]-i\int_{0}^{t}\mathbb E_x\biggl[(\varkappa \cdot b(X_s))w(t-s,X_s,\varkappa)\biggr]ds\\
&=\int_{\mathbb R^d} e^{i \varkappa \cdot y}e^{-t\Lambda(b)}(x,y) dy - i\int_0^t \int_{\mathbb R^d} e^{-s\Lambda(b)}(x,z)(\varkappa \cdot b(z))w(t-s,z,\varkappa)dzds.
\end{align*}

\smallskip

2.\,\textit{Set $\tilde{w}(t,x,\varkappa):=e^{i \varkappa \cdot x - t|\varkappa|^\alpha}$. This is another bounded solution to \eqref{eq_w}.}

Indeed, multiplying  the Duhamel formula
$$
e^{-t\Lambda}(x,y)=e^{-t(-\Delta)^{\frac{\alpha}{2}}}(x,y) + \int_0^t \langle e^{-(t-s)\Lambda}(x,\cdot) b(\cdot)\cdot \nabla_\cdot e^{-s(-\Delta)^{\frac{\alpha}{2}}}(\cdot,y)\rangle ds
$$
 (Corollary \ref{aux_prop}(\textit{iv}))
by $e^{i\varkappa \cdot y}$ and then integrating in $y$, we obtain the required.

\smallskip

Next, let us show that a bounded solution to \eqref{eq_w} is unique. We will need

\smallskip

\smallskip

3.\,\textit{For every $\varkappa \in \mathbb R^d$ there exists $T=T(\varkappa)>0$ such that the mapping
$$
(Hv)(t,x):=-i\int_0^t \int_{\mathbb R^d} e^{-(t-s)\Lambda(b)}(x,z)(\varkappa \cdot b(z))v(s,z)dsdz, \quad (t,x) \in [0,T] \times \mathbb R^d,
$$
is a contraction on $L^p(\mathbb R^d, |b|\eta^{-p+2}d\mathcal L^d; L^\infty[0,T])$ (i.e.\,functions taking values in $L^\infty[0,T]$) for $p$ as in Proposition \ref{weight_lem}.}

Indeed, we have
\begin{align*}
|Hv(t,x)| & \leq \left| \int_0^t \langle e^{-(t-s)\Lambda(b)}(x,\cdot)(\varkappa \cdot b(\cdot))v(s,\cdot)\rangle ds \right|  \\
& \leq |\varkappa|\int_0^t \langle e^{-(t-s)\Lambda(b)}(x,\cdot)|b(\cdot)|^{\frac{1}{p'}}|b(\cdot)|^{\frac{1}{p}}|v(s,\cdot)|\rangle ds \\
& \leq |\varkappa|\int_0^t \langle e^{-(t-s)\Lambda(b)}(x,\cdot)|b(\cdot)|^{\frac{1}{p'}}|b(\cdot)|^{\frac{1}{p}}\sup_{\tau \in [0,T]}|v(\tau,\cdot)|\rangle ds \label{H_est} \tag{$\ast$}
\end{align*}
Let us note that, for every $x \in \mathbb R^d$,
\begin{align*}
&|b(x)|^{\frac{1}{p}}\eta^{-1}(x)\sup_{t \in [0,T]}\int_0^t \langle e^{-(t-s)\Lambda(b)}(x,\cdot)|b(\cdot)|^{\frac{1}{p'}}\eta(\cdot)\rangle ds \\
& (\text{we are applying the Dominated Convergence Theorem})\\
& |b(x)|^{\frac{1}{p}}\eta^{-1}(x)\sup_{t \in [0,T]}\lim_m\int_0^t \langle e^{-(t-s)\Lambda(b)}(x,\cdot)|b_m(\cdot)|^{\frac{1}{p'}}\eta(\cdot)\rangle ds,
\end{align*}
where, in turn, the last term
\begin{align*}
& |b|^{\frac{1}{p}}\eta^{-1}\sup_{t \in [0,T]}\lim_m\int_0^t e^{-(t-s)\Lambda(b)}|b_m|^{\frac{1}{p'}}\eta ds \\
& \leq |b|^{\frac{1}{p}}\eta^{-1}e^{\mu T}\lim_m(\mu+\Lambda_{C_\infty}(b))^{-1}|b_m|^{\frac{1}{p'}}\eta \in \mathcal B(L^p_{\eta}) \quad \text{by Proposition \ref{weight_lem}\eqref{j_3_w}}.
\end{align*}
Also by Proposition \ref{weight_lem}\eqref{j_3_w},
selecting $\mu$ sufficiently large, and then selecting $T$ sufficiently small, the $L^p_\eta \rightarrow L^p_\eta$ norm of the last operator can be made arbitrarily small. Applying this in \eqref{H_est}, we obtain that 
$H$ is indeed a contraction on $L^p(\mathbb R^d, |b|\eta^{-p+2}d\mathcal L^d; L^\infty[0,T])$.

\medskip

We have $L^\infty([0,T] \times \mathbb R^d) \subset L^p(\mathbb R^d, |b|\eta^{-p+2}d\mathcal L^d; L^\infty[0,T])$ since $|b|\eta^{-p+2} \in L^1(\mathbb R^d)$ (Lemma \ref{p_eta_lem}). Combining the assertions of Steps 1-3, we obtain that for every $\varkappa \in \mathbb R^d$
$$
w(t,x,\varkappa)=\tilde{w}(t,x,\varkappa) \quad \text{ in } L^p(\mathbb R^d, |b|\eta^{-p+2}d\mathcal L^d; L^\infty[0,T]),
$$
and thus
$$
w(t,x,\varkappa)=\tilde{w}(t,x,\varkappa) \quad \text{ for a.e.\,$x \in \mathbb R^d$}
$$
(although $t<T(\varkappa)$, one can get rid of this constraint using the reproduction property of $e^{-t\Lambda(b)}$, so without loss of generality $T \neq T(\varkappa)$). Now, applying Corollary \ref{aux_prop}(\textit{iii}) to the RHS of \eqref{eq_w}, we obtain that for every $\varkappa \in \mathbb R^d$ $w(t,x,\varkappa)$ is continuous in $t$ and $x$, and so $w=\tilde{w}$ everywhere. Thus, for all $t \leq T$, $x \in \mathbb R^d$
$$
\mathbb E_x \biggl[e^{i \varkappa \cdot \bigl (X_t-X_0-\int_0^t b(X_s)ds\bigr) } \biggr] = e^{-\varkappa \cdot x}w(t,x,\varkappa)=e^{-t|\varkappa|^\alpha}.
$$
By a standard result, $Z_t$ is a symmetric $\alpha$-stable process.
The proof of Theorem \ref{ThCinfty}(\textit{vii}) is completed.

\appendix

\section{Pointwise bound on $\nabla \big(\mu+(-\Delta)^{\frac{\alpha}{2}}\big)^{-\gamma}$}

\label{app_est}

Set $A := (-\Delta)^{\frac{\alpha}{2}}$. 

\medskip


\noindent\textit{There exists constant $c_{d,\alpha,\gamma}  > 0$ such that}
\begin{equation}
\tag{A.1}
\label{estimation2}
\big|\nabla_y \big(\mu+A\big)^{-\gamma}(x,y)\big| \leq c_{d,\alpha,\gamma}(\mu +A)^{-\gamma+\frac{1}{\alpha}}(x,y), \quad \frac{1}{\alpha}<\gamma \leq 1, \quad \mu>0.
\end{equation}

\begin{proof}
We will need the well known estimates (see e.g.\,\cite{BJ}):
\begin{equation}
\label{F1}
\tag{A.2}
e^{-tA}(x,y) \geq C_{d,\alpha}  \biggl(t^{-\frac{d}{\alpha}} \wedge \frac{t}{|x-y|^{d+\alpha}}\biggr),
\end{equation}
\begin{equation}
\label{F2}
\tag{A.3}
 | \nabla_y e^{-tA}(x,y)|  \leq K_{d,\alpha} t^{-\frac{1}{\alpha}}\left(t^{-\frac{d}{\alpha}} \wedge \frac{t}{|x-y|^{d+\alpha}}\right).
\end{equation}
By the formula
\begin{equation*}
 (\mu+A)^{-\gamma}=\frac{1}{\Gamma (\gamma)}\int_{0}^{\infty}e^{-\mu t}t^{\gamma-1}e^{-tA}dt, 
\end{equation*}
we have
\begin{align*}
 \big|\nabla_y \big(\mu+A\big)^{-\gamma}(x,y)\big| &\leq  \frac{1}{\Gamma (\gamma)}  \int_{0}^{\infty}e^{-\mu t}t^{\gamma -1}| \nabla e^{-tA}(x,y) |dt \\
& (\text{we are using \eqref{F2} and then \eqref{F1}}) \\
 & \leq   K_{d,\alpha}C_{d,\alpha}\frac{1}{\Gamma (\gamma)} \int_{0}^{\infty}e^{-\mu t}t^{\gamma - \frac{1}{\alpha} -1} e^{-tA}(x,y) dt   \\
  & =  K_{d,\alpha}C_{d,\alpha} \frac{\Gamma(\gamma -\frac{1}{\alpha})}{\Gamma (\gamma)}\big(\mu+A\big)^{-\gamma+\frac{1}{\alpha}}(x,y).
\end{align*}
\end{proof}

\section{$L^p$ bounds for symmetric Markov generators}

\label{lp_sect}

Let $X$ be a set and $m$ a measure on $X$. In this section we use notation
$$
\langle h\rangle := \int_{X} hdm, \quad \langle h,g\rangle := \langle h\bar{g}\rangle.
$$

Let $-A$ be a symmetric Markov generator on $L^2 = L^2(X,m)$, i.e.\,a self-adjoint operator $A \geq 0$ that generates a contraction $C_0$ semigroup such that $e^{-tA}L^2_+ \subset L^2_+$ and for all $t>0$
$$
(f \in L^2, |f| \leq 1) \quad \Rightarrow \quad |e^{-tA}f| \leq 1.
$$
The semigroup $e^{-tA}$ determines a consistent family of contraction $C_0$ semigroups on $L^p$,
\begin{equation*}
e^{-tA_p} :=\left[ e^{-tA}\upharpoonright L^2 \cap L^p\right]_{p \rightarrow p}^{\rm clos} \in \mathcal{B}(L^{p}), \quad p \in [1,\infty[.
\end{equation*}


The proof of the next lemma is based on inequalities for symmetric Markov generators of \cite[Theorem 2.1]{LS} (see also \cite{BS}):


\begin{lemma}
\label{abc_lem}
Let $m$ be $\sigma$-finite and $1<\alpha \leq 2$. 
If $0 \leq V \in L^1_{\loc}$ satisfies
\begin{equation*}
\|V^{\frac{1}{2}}(\lambda+A)^{-\frac{\alpha-1}{2\alpha}}\|_{2 \rightarrow 2} \leq \sqrt{\delta} \quad \text{ for some } \lambda=\lambda_\delta,
\end{equation*}
then, for every $p \in ]1,\infty[$ and   $\mu \geq \lambda$, 

\smallskip

{\rm($\mathbf{a}$)}  $\|V^{\frac{1}{p}}(\mu +A_p)^{-\frac{\alpha-1}{\alpha}}\|_{p \rightarrow p}  \leq (\delta c_{p})^{\frac{1}{p}}\mu^{-\frac{\alpha-1}{\alpha p'}}, \quad c_p:=\frac{4}{pp'}$,

\smallskip

{\rm($\mathbf{b}$)} $\|V^{\frac{1}{p}}(\mu +A_p)^{-\frac{\alpha-1}{\alpha}}V^{\frac{1}{p'}}f\|_{p}  \leq \delta c_{p}\|f\|_p, \quad f \in \mathcal E:=\bigcup_{\varepsilon>0}e^{-\varepsilon |V|}L^p,$

\smallskip

{\rm($\mathbf{c}$)} $\|(\mu +A_p )^{-\frac{\alpha-1}{\alpha}}V^{\frac{1}{p'}}f\|_{p} \leq (\delta c_{p})^{\frac{1}{p'}}\mu^{-\frac{\alpha-1}{\alpha p}}\|f\|_p, \quad f \in \mathcal E.$

\end{lemma}


\begin{proof} 
(\textbf{a}) Let $E:=(\mu +A)^{\frac{\alpha-1}{\alpha}}$ in $L^2$. Since $-E$ is a symmetric
Markov generator, by \cite[Theorem 2.1]{LS} for every $p \in ]1, \infty [$
$$0 \leq u \in D(E_{p}) \quad \Rightarrow \quad u^{\frac{p}{2}} \in D(E^{\frac{1}{2}}), \quad \|E^{\frac{1}{2}}u^{\frac{p}{2}}\|_{2}^{2} \leq c_p\langle E_p u,u^{p-1} \rangle.
$$
Since $\mu \geq \lambda$, and so $\|V^{\frac{1}{2}}E^{-\frac{1}{2}}\|_{2 \rightarrow 2} \leq \sqrt{\delta}$ , we obtain for $u:=E_{p}^{-1}|f|$, $f \in L^p$
\begin{equation*}
\|V^{\frac{1}{2}} u^{\frac{p}{2}} \|_{2}^{2} \leq  \delta c_{p} \langle E_{p}u,u^{p-1} \rangle = \delta c_{p}\langle |f|,u^{p-1}\rangle.
\end{equation*}
It follows that
\begin{align*}
\|V^{\frac{1}{p}}u\|_{p}^{p}& = \|V^{\frac{1}{2}} u^{\frac{p}{2}} \|_{2}^{2}
 \leq  \delta c_{p}\|f\|_{p}\|u\|_{p}^{p-1} \leq  \delta c_{p} \mu^{-\frac{\alpha-1}{\alpha}(p-1)}\|f\|_{p}^{p}.
\end{align*}
So, $\|V^{\frac{1}{p}}E_{p}^{-1}f\|_{p}^{p}  \leq  \delta c_{p}\mu^{-\frac{\alpha-1}{\alpha}(p-1)}\|f\|_{p}^{p}$,
which yields (\textbf{a}).

\smallskip

($\textbf{b}$) We argue as in ($\textbf{a}$). Put $u:=E_{p}^{-1}V^{\frac{1}{p'}}|f|$, $f \in \mathcal E$.
Then
\begin{align*}
\|V^{\frac{1}{p}}u\|_{p}^{p} \leq   \delta c_{p} \langle E_{p}u,u^{p-1} \rangle 
= \delta c_{p} \langle V^{\frac{1}{p'}} |f|,u^{p-1} \rangle
\leq  \delta c_{p} \|V^{\frac{1}{p}}u\|_{p}^{p-1}\|f\|_{p}
\end{align*}
so $\|V^{\frac{1}{p}}u\|_{p}\leq  \delta c_{p} \|f\|_{p}$, and thus
$\|V^{\frac{1}{p}}E_{p}^{-1}V^{\frac{1}{p'}}f\|_{p}^{p} \leq  \delta c_{p} \|f\|_{p}$.

\smallskip

($\textbf{c}$) follows from (\textbf{a})  by duality.
\end{proof}

\section{Proof of uniqueness (Remark \ref{rem2}) }
\label{u_app}

Denote $A:=(-\Delta)^{\frac{\alpha}{2}}$.

\smallskip

Step 1. Let us show that for every $f \in C_c^\infty$, $(\mu+\Lambda_{C_\infty}(b))^{-1} f = R_\mu^Q f$ $\mathcal L^d$ a.e.\,on  $\mathbb R^d$.
Indeed, since $\{\mathbb Q_x\}$ is a weak solution to \eqref{eq0}, we have by It\^{o}'s formula 
$$
(\mu+A)^{-1} h= R_\mu^Q[(1-b\cdot\nabla (\mu+A)^{-1})h], \quad h \in C_c^\infty.
$$
Since $b \in \mathbf{F}_\delta^{\scriptscriptstyle \frac{\alpha-1}{2}}$, we have $\|1-b\cdot\nabla(\mu+A)^{-1}\|_{\mathcal W^{-\frac{\alpha-1}{2},2} \rightarrow \mathcal W^{-\frac{\alpha-1}{2},2}} <1+\delta$.
By our assumption, $R_\mu^Q$ extends by continuity to $R_{\mu,2}^Q \in \mathcal B(\mathcal W^{-\frac{\alpha-1}{2},2},L^2)$.
Thus, 
$$
(\mu+A)^{-1} g= R_{\mu,2}^Q[(1-b\cdot\nabla (\mu+A)^{-1})g], \quad g \in \mathcal W^{-\frac{\alpha-1}{2},2}.
$$
Take $g=(1-b\cdot\nabla (\mu+A)^{-1})^{-1} f \in \mathcal W^{-\frac{\alpha-1}{2},2}$, $f \in C_c^\infty$. Then by the construction of $\Lambda_2(b)$ (cf.\,Theorem \ref{thmL2}),
$
(\mu+\Lambda_2(b))^{-1}f=R_{\mu,2}^Q f
$ in $L^2$, and the assertion of Step 1 follows since  $(\mu+\Lambda_{C_\infty}(b))^{-1} \upharpoonright C_c^\infty =(\mu+\Lambda_2(b))^{-1} \upharpoonright C_c^\infty$.

\smallskip

Step 2. Since by our assumption
$R_\mu^Q f$ is continuous on $\mathbb R^d$, Step 1 yields that $(\mu+\Lambda_{C_\infty}(b))^{-1} f = R_\mu^Q f$ everywhere on  $\mathbb R^d$, as claimed.

\section{Trotter's Approximation Theorem}

\label{app_trotter}

\begin{theorem}
Let $e^{-t\Lambda_{k}}$, $k=1,2,\dots$, be a sequence of $C_0$ semigroups on a (complex) Banach space $Y$.
Assume that 

{\rm (\textit{i})}~$\sup_{k}\|(\mu+\Lambda_{k})\|_{Y \rightarrow Y} \leq \mu^{-1}$\;$\forall\, \mu > 0$, or 
$\sup_{k}\|(z+\Lambda_{k})\|_{Y \rightarrow Y} \leq C |z|^{-1}$\;\;$\forall\,z \text{ with } \Real z > 0 $.

{\rm (\textit{ii})}~$s\mbox{-}Y\mbox{-}\lim_{\mu \uparrow \infty }\mu (\mu +  \Lambda_{k})^{-1}=1$ uniformly in $k$.

{\rm (\textit{iii})}~ $s\mbox{-}Y\mbox{-}\lim_{k} (z +\Lambda_{k})^{-1}$ exists for some $z$ with $\Real z > 0 $.

Then there is a  $C_{0}$ semigroup  $e^{-t\Lambda}$ on $Y$ such that
$$(z+\Lambda_{k})^{-1} \stackrel{s}{\rightarrow}(z+\Lambda)^{-1} \quad \text{ in } Y, \quad \forall\,z \text{ with } \Real z > 0,$$ and  $$ e^{-t\Lambda_{k}}\stackrel{s}{\rightarrow}e^{-t\Lambda} \quad \text{ in } Y \quad \text{loc.\,uniformly in $t \geq 0$}.$$ 
\end{theorem}

(This is a special case of the Trotter Approximation Theorem  \cite[Ch.\,IX, sect.\,2]{Ka}.)

\end{document}